\documentclass[11pt, reqno]{amsart}
\usepackage{amsfonts,latexsym,enumerate}
\usepackage{amsmath}
\usepackage{amscd}
\usepackage{float,amsmath,amssymb,mathrsfs,bm,multirow,graphics}
\usepackage[dvips]{graphicx}
\usepackage[percent]{overpic}
\usepackage{amsaddr}
\usepackage[numbers,sort&compress]{natbib}
\usepackage{todonotes}

\usepackage{tikz}

\newcommand{\R}{{\Bbb R}}

\newcommand{\C}{{\Bbb C}}
\newcommand{\D}{{\Bbb D}}

\newcommand{\diag}{\text{\upshape diag\,}}
\newcommand{\re}{\text{\upshape Re\,}}

\DeclareMathOperator{\dist}{dist}

\def\XXint#1#2#3{{\setbox0=\hbox{$#1{#2#3}{\int}$}
\vcenter{\hbox{$#2#3$}}\kern-.5\wd0}}

\newtheorem{theorem}{Theorem}
\newtheorem{proposition}{Proposition}[section]

\newtheorem{lemma}[proposition]{Lemma}
\newtheorem{definition}[proposition]{Definition}
\newtheorem{assumption}[proposition]{Assumption}
\newtheorem{remark}[proposition]{Remark}

\newtheorem{figuretext}[proposition]{Figure}
\newtheorem{RHproblem}[proposition]{RH problem}

\numberwithin{equation}{section}

\usepackage[colorlinks=true]{hyperref}
\hypersetup{urlcolor=blue, citecolor=red, linkcolor=blue}

\input epsf
\title[Long-time asymptotics for an integrable evolution equation]
{Long-time asymptotics for an integrable evolution equation with a $3 \times 3$ Lax pair}

\author{C. Charlier and J. Lenells}
\address{Department of Mathematics, KTH Royal Institute of Technology, \\ 100 44 Stockholm, Sweden.}
\email{cchar@kth.se}
\email{jlenells@kth.se}

\begin{document}
\begin{abstract}
We derive a Riemann--Hilbert representation for the solution of an integrable nonlinear evolution equation with a $3 \times 3$ Lax pair. We use the derived representation to obtain formulas for the long-time asymptotics.
\end{abstract}

\maketitle




\section{Introduction}
We consider the nonlinear evolution equation
\begin{align}\label{3x3eq}
& iq_t - \frac{1}{\sqrt{3}}q_{xx} + 2\sqrt{3}\bar{q}\bar{q}_x =0,  
\end{align}
where $q(x,t)$  is a complex-valued function of a space variable $x \in \R$  and a time variable $t \geq 0$.
Equation (\ref{3x3eq}) is a type of quadratic derivative nonlinear Schr\"odinger (NLS) equation: instead of the cubic nonlinearity $|q|^2q$ of the standard NLS equation, it features the quadratic nonlinearity $\bar{q}\bar{q}_x$. One of the interesting properties of equation (\ref{3x3eq}) is that it is integrable. The integrability of (\ref{3x3eq}) is a consequence of the fact that it is a reduction of a system that was studied and shown to be integrable in \cite{L3x3}.
 While the Lax pair for the NLS equation involves $2 \times 2$ matrices, the Lax pair for (\ref{3x3eq}) involves $3 \times 3$ matrices; this corresponds to the fact that the isospectral problem for (\ref{3x3eq}) is a third-order rather than a second-order differential equation.

The simple structure of equation (\ref{3x3eq}) in combination with its integrability suggests that it could be relevant for various applications (the wide applicability of integrable equations has been given a heuristic explanation by Calogero \cite{C1991}). However, we will not attempt to find such applications here; instead we adopt the point of view that (\ref{3x3eq}) is sufficiently interesting as a prototypical example of a nonlinear PDE with a $3 \times 3$ matrix Lax pair to warrant further study. 
In fact, within the class of integrable equations with $3 \times 3$ matrix Lax pairs---which includes equations such as the Boussinesq \cite{B1872}, Sasa--Satsuma \cite{SS1991}, Sawada--Kotera \cite{SK1974}, Kaup--Kupershmidt \cite{K1980}, Degasperis--Procesi \cite{DP1999}, and $3$-wave resonant interaction \cite{ZM1973} equations---equation (\ref{3x3eq}) is one of the simplest and most natural generalizations of the NLS equation, at least from the point of view of the Lax pair. Indeed, while the $x$ and $t$ parts of the Lax pair for NLS involve the matrices $k\sigma_3$ and $k^2 \sigma_3$, respectively, where $\sigma_3$ is the diagonal matrix whose entries are the second roots of unity, i.e., $\sigma_3 = \diag(1,-1)$, the $x$ and $t$ parts of (\ref{3x3eq}) involve the matrices $kJ$ and $k^2 J^2$, where $J$ is the diagonal matrix whose entries are the third roots of unity, i.e., $J = \mbox{diag}(\omega, \omega^{2}, 1)$ with $\omega = e^{2\pi i/3}$. The Lax pairs of NLS and (\ref{3x3eq}) also exhibit other similarities. This suggests that (\ref{3x3eq}) constitutes a prototypical example within the above class which is valuable as a mathematical playground. 


In this paper, we first develop an inverse scattering transform formalism for the solution of the initial-value problem for (\ref{3x3eq}). We then use steepest descent techniques to describe the asymptotic behavior of the solution for large times. More precisely, assuming that $q(x,t)$ is a Schwartz class solution of \eqref{3x3eq} with prescribed initial data $q_0(x) = q(x,0)$ at time $t = 0$, we show that $q(x,t)$ can be expressed in terms of the solution $m(x,t,k)$ of a $3 \times 3$ matrix Riemann--Hilbert (RH) problem. The jump contour $\Gamma$ of this RH problem consists of the three lines $\R \cup \omega \R \cup \omega^2 \R$ where $\omega = e^{2\pi i/3}$, see Figure \ref{Gamma.pdf}. The jump matrices are expressed in terms of two reflection coefficients, $r_{1}(k)$ and $r_{2}(k)$, which are defined in terms of the initial data $q_0(x)$ via the solution of linear integral equations (the direct scattering problem). The RH problem provides the solution of the inverse scattering problem in the sense that the value of $q(x,t)$ at any later time $t > 0$ can be recovered from $m(x,t,k)$ via the relation
$$q(x,t) = \lim_{k\to \infty} \big(k m(x,t,k)\big)_{13}.$$
The large $t$ behavior of $m(x,t,k)$ can be extracted by performing a Deift-Zhou \cite{DZ1993} steepest descent analysis of the RH problem and this leads to an asymptotic formula for $q(x,t)$. 

Our main results are stated in Theorem \ref{RHth}. The main difficulty in the proof of Theorem \ref{RHth} lies in the fact that the Lax pair---and hence also the RH problem---involves $3\times 3$ matrices. This means that the implementation of the direct and inverse scattering, as well as the steepest descent analysis of the RH problem, in general requires substantial effort. However, it turns out that the Lax pair derived in \cite{L3x3} for \eqref{3x3eq} is similar in structure to the Lax pair of the so-called ``good'' Boussinesq equation, and this  observation will save us a lot of effort. Indeed, the direct and inverse scattering problems for the ``good'' Boussinesq equation were studied in detail in  \cite{CL} and a steepest descent analysis of the corresponding RH problem was carried out in \cite{CLW}. Consequently, by formulating our problems appropriately, we will be able to recycle many of the results and proofs from \cite{CL, CLW}. In particular, the steepest descent analysis of \cite{CLW} can be applied almost unchanged to obtain asymptotics also for the solution of (\ref{3x3eq}). The upshot is that by exploiting the connection to the ``good'' Boussinesq equation and drawing on \cite{CL, CLW}, we can obtain results for (\ref{3x3eq}) with relatively little effort.

Other works analyzing the long-time asymptotics for integrable evolution equations with $3 \times 3$  Lax pairs include \cite{HL2020, LGX2018, LG2019} (for the Sasa--Satsuma equation), \cite{CLW} (for the ``good'' Boussinesq equation), and \cite{BS2013, BLS2017} (for the Degasperis--Procesi equation).

Our main theorem is stated in Section \ref{mainsec}. The direct scattering problem is considered in Section \ref{specsec}. The proof of the main theorem is presented in Section \ref{proofsec}.

\begin{figure}
\begin{center}
 \begin{overpic}[width=.4\textwidth]{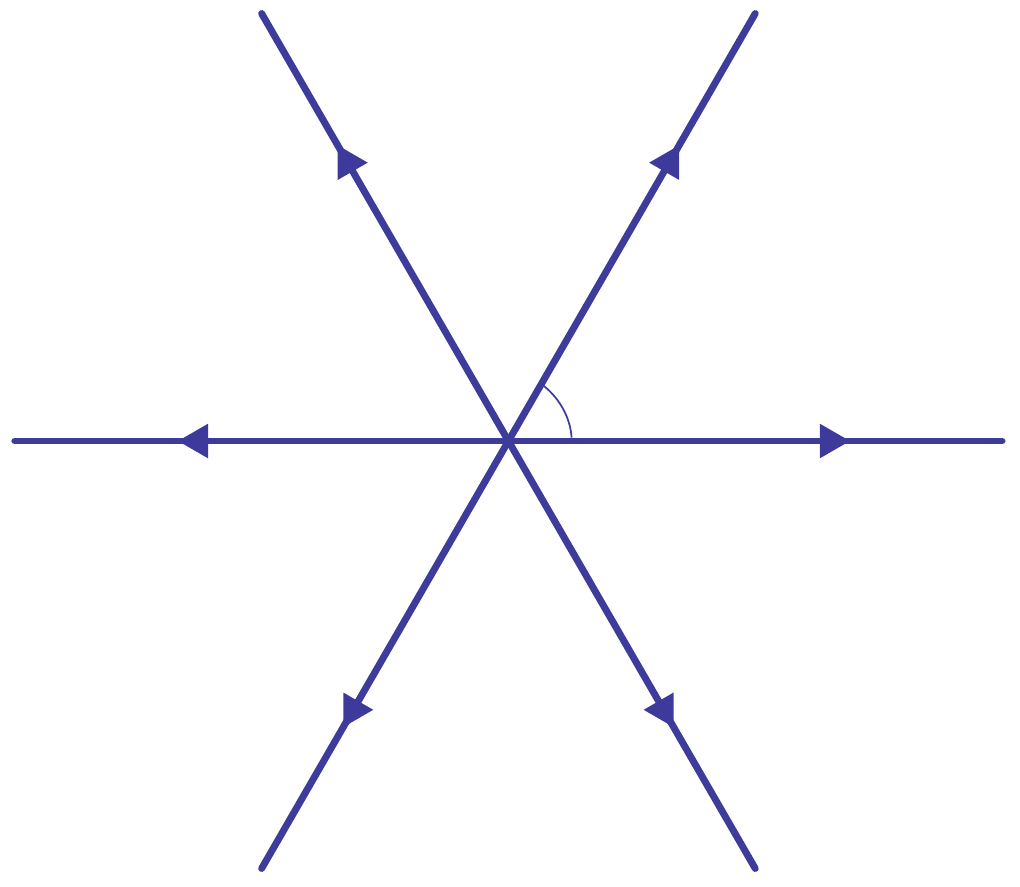}
  \put(101,41.5){\small $\Gamma$}
 \put(56,47){\small $\pi/3$}
 \put(80,60){\small $D_1$}
 \put(48,74){\small $D_2$}
 \put(17,60){\small $D_3$}
 \put(17,25){\small $D_4$}
 \put(48,12){\small $D_5$}
 \put(80,25){\small $D_6$}
  \put(81,38){\small $1$}
 \put(67.6,68){\small $2$}
 \put(29,68){\small $3$}
 \put(17,38){\small $4$}
 \put(30,17){\small $5$}
 \put(67,17){\small $6$}
   \end{overpic}
     \begin{figuretext}\label{Gamma.pdf}
       The contour $\Gamma$ and the open subsets $\{D_n\}_1^6$ of the complex $k$-plane.
     \end{figuretext}
     \end{center}
\end{figure}

\section{Main results}\label{mainsec}
Let $\mathcal{S}(\R)$ denote the Schwartz class of rapidly decreasing functions on the real line, and let $q_0 \in \mathcal{S}(\R)$ be some given initial data for (\ref{3x3eq}). 
The RH problem for $m$ is formulated in terms of two spectral functions $r_1(k)$  and $r_2(k)$, which depend only on $q_0$. The functions $r_1(k)$  and $r_2(k)$ are referred to as reflection coefficients and are defined as follows. Let $\omega = e^{2\pi i/3}$ and define $\{l_j(k), z_j(k)\}_{j=1}^3$ by
\begin{align}\label{ljzjdef}
&l_j(k) = \omega^j k, \quad z_j(k) = \omega^{2j} k^{2}, \qquad k \in \C.
\end{align}
Let the matrix-valued function $\mathcal{U}(x)$ be given by
\begin{align}\label{mathsfUdef intro}
\mathcal{U}(x) = &\; \begin{pmatrix} 0 & (1-\omega^2) \overline{q_0(x)} & (1-\omega)q_0(x) \\
(1-\omega)q_0(x) & 0 & (1-\omega^2)\overline{q_0(x)} \\
(1-\omega^2)\overline{q_0(x)} & (1-\omega)q_0(x) & 0 \end{pmatrix}.
\end{align} 
Define the $3 \times 3$-matrix valued eigenfunctions $X(x,k)$ and $X^A(x,k)$ as the unique solutions of the Volterra integral equations
\begin{subequations}\label{XXAdef intro}
\begin{align}  
 & X(x,k) = I - \int_x^{\infty} e^{(x-x')\widehat{\mathcal{L}(k)}} \big(\mathcal{U}(x')X(x',k)\big) dx',
	\\
 & X^A(x,k) = I + \int_x^{\infty} e^{-(x-x')\widehat{\mathcal{L}(k)}} \big(\mathcal{U}(x')^T X^A(x',k)\big) dx',	
\end{align}
\end{subequations}
where $\mathcal{L} = \diag(l_1 , l_2 , l_3)$, $\widehat{\mathcal{L}}$ denotes the operator which acts on a $3 \times 3$ matrix $B$ by $\widehat{\mathcal{L}}B = [\mathcal{L}, B]$ (and thus $e^{\widehat{\mathcal{L}}}B = e^\mathcal{L} B e^{-\mathcal{L}}$), and $\mathcal{U}^T$ denotes the transpose of $\mathcal{U}$. Define $s(k)$ and $s^A(k)$ by 
\begin{align}\label{sdef}
& s(k) = I - \int_\R e^{-x\widehat{\mathcal{L}(k)}} \big(\mathcal{U}(x) X(x,k)\big)dx,
 	\\ \label{sAdef}
& s^A(k) = I + \int_\R e^{x\widehat{\mathcal{L}(k)}} \big(\mathcal{U}(x)^T X^A(x,k)\big)dx.
\end{align}
The reflection coefficients $r_1(k)$  and $r_2(k)$ are defined by
\begin{align}\label{3x3r1r2def}
\begin{cases}
r_1(k) = \frac{(s(k))_{12}}{(s(k))_{11}}, & k \in (0,\infty),
	\\ 
r_2(k) = \frac{(s^A(k))_{12}}{(s^A(k))_{11}}, \quad & k \in (-\infty,0).
\end{cases}
\end{align}	
We note that the definition (\ref{3x3r1r2def}) of $r_1(k)$  and $r_2(k)$ coincides with the analogous definition in \cite{CL} except that the matrix $\mathsf{U}$ of \cite{CL} has been replaced by the matrix $\mathcal{U}$ defined in \eqref{mathsfUdef intro}.

Let $D_{1},\ldots,D_{6}$ be the open sectors shown in Figure \ref{Gamma.pdf}. The functions $(s(k))_{11}$ and $(s^A(k))_{11}$, which appear in the denominators of \eqref{3x3r1r2def}, are analytic in $D_{1}$ and $D_{4}$, respectively. The case when $(s(k))_{11}$ and $(s^A(k))_{11}$ have a finite number of simple zeros in $D_1$ and $D_4$ (related to the presence of solitons) can be handled as in \cite{L3x3}, but to avoid unnecessary technicalities, we will assume that no such zeros are present.
 
\begin{assumption}[Absence of solitons]\label{solitonlessassumption}\upshape
Assume that $(s(k))_{11}$ and $(s^A(k))_{11}$ are nonzero for $k \in \bar{D}_1$ and $k \in \bar{D}_4$, respectively.
\end{assumption}

The next lemma, whose proof is given in Section \ref{r1r2subsec}, establishes several important properties of $r_{1}(k)$ and $r_{2}(k)$.

\begin{lemma}[Properties of $r_1(k)$ and $r_2(k)$]\label{r1r2lemma}
Suppose $q_0 \in \mathcal{S}(\R)$ is such that Assumption \ref{solitonlessassumption} holds. Then $r_1:(0,\infty) \to \C$ and $r_2:(-\infty,0) \to \C$ are well-defined by (\ref{3x3r1r2def}) and have the following properties:
\begin{enumerate}[$(i)$]
 \item $r_1 \in C^\infty((0,\infty))$ and $r_2 \in C^\infty((-\infty,0))$. 
 
 \item The functions $r_1(k)$, $r_2(k)$, and their derivatives $\partial_{k}^{j}r_{\ell}(k)$ have continuous boundary values at $k=0$ for $\ell = 1,2$ and $j = 0,1,2,\ldots$, and there exist expansions
 \begin{subequations}\label{r1r2atzero}
\begin{align}
& r_{1}(k) = r_{1}(0) + r_{1}'(0)k + \tfrac{1}{2}r_{1}''(0)k^{2} + \cdots, & & k \to 0, \ k >0, \\
& r_{2}(k) = r_{2}(0) + r_{2}'(0)k + \tfrac{1}{2}r_{2}''(0)k^{2} + \cdots, & & k \to 0, \ k <0,
\end{align}
\end{subequations}
which can be differentiated termwise any number of times.

\item $r_1(k)$ and $r_2(k)$ are rapidly decreasing as $|k| \to \infty$, i.e.,
\begin{subequations}\label{r1r2rapiddecay}
\begin{align}\label{r1rapiddecay}
& \max_{j=0,1,\dots,N}\sup_{k \in (0,\infty)} (1+|k|)^N |\partial_k^jr_1(k)| < \infty,  
	\\
& \max_{j=0,1,\dots,N} \sup_{k \in (-\infty, 0)} (1+|k|)^N|\partial_k^jr_2(k)| < \infty,
\end{align}
\end{subequations}
for each integer $N \geq 0$. 
\item $|r_{1}(k)|<1$ for all $k \geq 0$ and $|r_{2}(k)| <1$ for all $k \leq 0$.
\end{enumerate} 
\end{lemma}

We will show that the solution $q(x,t)$ of (\ref{3x3eq}) with initial data $q(x,0) = q_0(x)$ can be recovered from the solution $m(x,t,k)$ of a $3 \times 3$ RH problem whose jumps are expressed in terms of $r_1$ and $r_2$. To state the RH problem for $m$, let $\Gamma$ be the contour consisting of the three lines $\R \cup \omega \R \cup \omega^2 \R$ oriented away from the origin as in Figure \ref{Gamma.pdf}.
For $1 \leq i \neq j \leq 3$, define $\theta_{ij} = \theta_{ij}(x,t,k)$ by
$$\theta_{ij}(x,t,k) = (l_i - l_j)x + (z_i - z_j)t.$$
Define the jump matrix $v(x,t,k)$ for $k \in \Gamma$ by
\begin{align}\nonumber
&  v_1 = 
  \begin{pmatrix}  
 1 & - r_1(k)e^{-\theta_{21}} & 0 \\
  \overline{r_1(k)}e^{\theta_{21}} & 1 - |r_1(k)|^2 & 0 \\
  0 & 0 & 1
  \end{pmatrix},
&&  v_2 = 
  \begin{pmatrix}   
 1 & 0 & 0 \\
 0 & 1 - |r_2(\omega k)|^2 & -\overline{r_2(\omega k)}e^{-\theta_{32}} \\
 0 & r_2(\omega k)e^{\theta_{32}} & 1 
    \end{pmatrix},
   	\\ \nonumber
  &v_3 = 
  \begin{pmatrix} 
 1 - |r_1(\omega^2 k)|^2 & 0 & \overline{r_1(\omega^2 k)}e^{-\theta_{31}} \\
 0 & 1 & 0 \\
 -r_1(\omega^2 k)e^{\theta_{31}} & 0 & 1  
  \end{pmatrix}, &&  v_4 = 
  \begin{pmatrix}  
  1 - |r_2(k)|^2 & -\overline{r_2(k)} e^{-\theta_{21}} & 0 \\
  r_2(k)e^{\theta_{21}} & 1 & 0 \\
  0 & 0 & 1
   \end{pmatrix},
   	\\ \label{vdef}
&  v_5 = 
  \begin{pmatrix}
  1 & 0 & 0 \\
  0 & 1 & -r_1(\omega k)e^{-\theta_{32}} \\
  0 & \overline{r_1(\omega k)}e^{\theta_{32}} & 1 - |r_1(\omega k)|^2
  \end{pmatrix},
&& v_6 = 
  \begin{pmatrix} 
  1 & 0 & r_2(\omega^2 k)e^{-\theta_{31}} \\
  0 & 1 & 0 \\
  -\overline{r_2(\omega^2 k)}e^{\theta_{31}} & 0 & 1 - |r_2(\omega^2 k)|^2
   \end{pmatrix},
\end{align}
where $v_j$ denotes the restriction of $v$  to the part of $\Gamma$ labeled by $j$ in Figure \ref{Gamma.pdf}.

\begin{RHproblem}[RH problem for $m$]\label{RH problem for m} Given $r_1:(0,\infty) \to \C$ and $r_2:(-\infty,0) \to \C$, find a $3 \times 3$-matrix valued function $m(x,t,k)$ with the following properties:
\begin{enumerate}[$(a)$]
\item $m(x,t,\cdot) : \C\setminus \Gamma \to \C^{3 \times 3}$ is analytic.

\item The limits of $m(x,t,k)$ as $k$ approaches $\Gamma\setminus \{0\}$ from the left $(+)$ and right $(-)$ exist, are continuous on $\Gamma\setminus \{0\}$, and are related by
\begin{align}\label{mjumpcondition}
& m_{+}(x,t,k) = m_{-}(x,t,k)v(x,t,k), \qquad k \in \Gamma \setminus \{0\},
\end{align}
where $v$ is defined in terms of $r_{1}$ and $r_{2}$ by \eqref{vdef}.

\item $m(x,t,k) = I + O(k^{-1})$ as $k \to \infty$, $k \notin \Gamma$.

\item $m(x,t,k) = O(1)$ as $k \to 0$.
\end{enumerate}
\end{RHproblem}

Before stating our main result, we introduce the notion of a Schwartz class solution. 

\begin{definition}\label{def: Schwartz class solution}\upshape
We call $q(x,t)$ a {\it Schwartz class solution of \eqref{3x3eq} with existence time $T \in (0, \infty]$ and initial data $q_0 \in \mathcal{S}(\R)$} if
\begin{enumerate}[$(i)$] 
  \item $q$ is a smooth complex-valued function of $(x,t) \in \R \times [0,T)$.

\item $q$ satisfies \eqref{3x3eq} for $(x,t) \in \R \times [0,T)$ and 
$$q(x,0) = q_0(x), \qquad x \in \R.$$ 

  \item $q$ has rapid decay as $|x| \to \infty$ in the sense that, for each integer $N \geq 1$,
$$\sup_{\substack{x \in \R \\ t \in [0, T)}} \sum_{i =0}^N (1+|x|)^N \, |\partial_x^i q(x,t)| < \infty.$$
\end{enumerate} 
\end{definition}

We can now state our main theorem. The first part of the theorem establishes a representation for the solution $q(x,t)$ of (\ref{3x3eq}) in terms of the solution of a RH problem; the second and third parts establish formulas for the long-time asymptotics of $q(x,t)$ for $x/t$ in compact subsets of $(0,+\infty)$ and $(-\infty,0)$, respectively, see Figure \ref{sectors.pdf}.

\begin{figure}
\bigskip
\begin{center}
 \begin{overpic}[width=.7\textwidth]{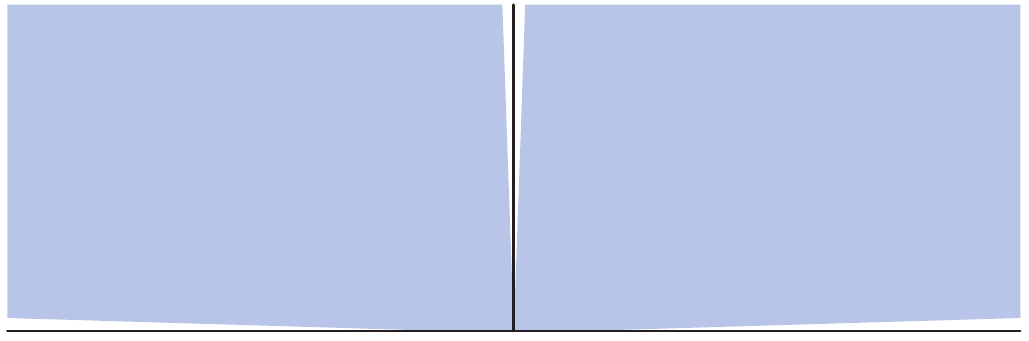}
  \put(101,0){\small $x$}
  \put(49.5,34.2){\small $t$}
   \end{overpic}
     \begin{figuretext}\label{sectors.pdf}
      The shaded areas represent the sectors in the $(x,t)$-plane in which the asymptotic formulas of Theorem \ref{RHth} apply. 
     \end{figuretext}
     \end{center}
\vspace{-.2cm}
\end{figure}

\begin{theorem}\label{RHth}
Suppose $q(x,t)$ is a Schwartz class solution of (\ref{3x3eq}) with existence time $T \in (0, \infty]$ and initial data $q_0 \in \mathcal{S}(\R)$ such that Assumption \ref{solitonlessassumption} holds. Define the spectral functions $r_1(k)$ and $r_2(k)$ in terms of $q_0$ by (\ref{3x3r1r2def}).

\begin{enumerate}[$(i)$]
\item The RH problem \ref{RH problem for m} has a unique solution $m(x,t,k)$ for each $(x,t) \in \R \times [0,T)$ and the formula
\begin{align}\label{qlim}
q(x,t) = \lim_{k\to \infty} \big(k \, m(x,t,k)\big)_{13}
\end{align}
expressing $q(x,t)$ in terms of $m(x,t,k)$ is valid for all $(x,t) \in \R \times [0,T)$.

\item As $t \to \infty$, we have
\begin{align*}
q(x,t) = &\; \frac{\sqrt{\pi} e^{-\frac{\pi \nu}{2}}}{3^{1/4} \sqrt{t}} 
\frac{e^{i\phi}}{r_1(k_0) \Gamma(i\nu)}
 + O(t^{-1}\ln t),
\end{align*}
uniformly for $\zeta=x/t$ in compact subsets of $(0,\infty)$, where $\Gamma$ denotes the Gamma function,
\begin{align*}
k_{0} = k_{0}(\zeta) = \frac{\zeta}{2}, \qquad \nu = \nu(\zeta) = -\frac{1}{2\pi}\ln(1-|r_{1}(k_{0})|^{2}) \geq 0,
\end{align*}
and
\begin{align*}
\phi = \phi(\zeta) = \frac{11\pi}{12}   + \nu \ln(6\sqrt{3}tk_0^2) - \sqrt{3} k_0^2 t
+ \frac{1}{\pi} \int_{k_{0}}^{+\infty} \ln\bigg|\frac{s-k_0}{s - \omega k_0}\bigg| d\ln(1-|r_1(s)|^{2}).
\end{align*}
\item As $t \to \infty$, we have
\begin{align*}
q(x,t) = &\; -\frac{\sqrt{\pi} e^{-\frac{\pi \tilde{\nu}}{2}}}{3^{1/4} \sqrt{t}} 
\frac{e^{i\tilde{\phi}}}{\tilde{r}_2(k_0) \Gamma(i\tilde{\nu})}
 + O(t^{-1}\ln t),
\end{align*}
uniformly for $\zeta=x/t$ in compact subsets of $(-\infty,0)$, where 
\begin{align*}
k_{0} = k_{0}(\zeta) = \frac{\zeta}{2}, \quad \tilde{\nu} = \tilde{\nu}(\zeta) = -\frac{1}{2\pi}\ln(1-|r_{2}(k_{0})|^{2}) \geq 0, \quad \tilde{r}_2(k_0) = \frac{(s^A(k_{0}))_{21}}{(s^A(k_{0}))_{11}},
\end{align*}
and
\begin{align*}
\tilde{\phi} = \tilde{\phi}(\zeta) = \frac{11\pi}{12}   + \tilde{\nu} \ln(6\sqrt{3}tk_0^2) - \sqrt{3} k_0^2 t
+ \frac{1}{\pi} \int_{-\infty}^{k_{0}} \ln\bigg|\frac{s-k_0}{s - \omega k_0}\bigg| d\ln(1-|r_2(s)|^{2}).
\end{align*}
\end{enumerate}
\end{theorem}
\begin{proof}
See Section \ref{proofsec}.
\end{proof}

\begin{remark}\upshape
Theorem \ref{RHth} provides the asymptotics of the solution $q(x,t)$ of  (\ref{3x3eq}) everywhere in the half-plane $t > 0$ except in narrow wedges containing the $x$ and $t$-axes, see Figure \ref{sectors.pdf}. 
The asymptotics in these wedges can also be analyzed via a steepest descent analysis of the RH problem for $m$, but this falls outside the scope of the present paper. As $\zeta = x/t \to 0$, the critical points in the steepest descent analysis merge at the origin, leading to qualitatively different asymptotics in the wedge containing the $t$-axis. Near the $x$-axis, the solution has decay consistent with the decay of the Schwartz class initial data.
\end{remark}

\begin{remark}\upshape
For many integrable PDEs with $2\times 2$ matrix Lax pairs, such as the KdV and mKdV equations, the reflection coefficient is bounded above by $1$ for all $k \in \R$. For these equations, this property is a direct consequence from the fact that the scattering matrix, which is of size $2 \times 2$, has unit determinant and satisfies a certain symmetry. For integrable PDEs with $3 \times 3$ matrix Lax pairs, the situation is considerably richer, because there is in general no simple argument that ensures an upper bound for the reflection coefficients. 
In Appendix \ref{section:appendix}, we explore this issue in the context of equation (\ref{3x3eq}).
Under Assumption \ref{solitonlessassumption}, we have shown in part $(iv)$ of Lemma \ref{r1r2lemma} that $|r_{1}(k)|<1$ for $k \geq 0$ and $|r_{2}(k)| <1$ for $k \leq 0$. We show in Appendix \ref{section:appendix} that these bounds on $r_1$ and $r_2$ hold even without Assumption \ref{solitonlessassumption} in the case of compactly supported initial data. (For general initial data $q_{0} \in \mathcal{S}(\mathbb{R})$, we establish  the slightly weaker bounds $|r_{1}(k)|\leq 1$ for $k \geq 0$ and $|r_{2}(k)|\leq 1$ for $k \leq 0$.) This is in contrast to other equations with $3 \times 3$ Lax pairs such as the ``good" Boussinesq equation, for which the absolute values of the reflection coefficients are, in general, not bounded above by $1$ unless further assumptions are imposed, see \cite{CL}. 
\end{remark}

\begin{remark}\upshape
Assumption \ref{solitonlessassumption} can be relaxed slightly in the case of compactly supported initial data. 
More precisely, it follows from the arguments in Appendix \ref{section:appendix} that $(s(k))_{11}$ and $(s^A(k))_{11}$ are automatically nonzero for $k \in \omega^{2}(-\infty,0] \subset \bar{D}_1$ and $k  \in \omega^{2}[0,\infty) \subset \bar{D}_4$, respectively, if $q_0$ has compact support. 
\end{remark}

\begin{remark}\upshape
The proof of the asymptotic formula valid for $\zeta = x/t > 0$ given in part $(ii)$ of Theorem \ref{RHth} relies on a steepest descent analysis of the RH problem for $m$. The asymptotics for $\zeta < 0$ can be derived in a similar way. However, for simplicity we choose to instead derive the asymptotics for $\zeta < 0$ from the asymptotics for $\zeta > 0$ and the fact that equation (\ref{3x3eq}) is invariant under the symmetry $q(x,t) \to -q(-x,t)$, see Section \ref{proofofiii}.
\end{remark}

\section{Direct scattering}\label{specsec}

\subsection{Lax pair}\label{Section: Lax pair}
Let $\omega := e^{2\pi i/3}$, define $\{l_j(k), z_j(k)\}_{j=1}^3$ by \eqref{ljzjdef}, and define $J$, $\mathcal{L} = \mathcal{L}(k)$, and $\mathcal{Z} = \mathcal{Z}(k)$ by
\begin{align*}
J = \mbox{diag}(\omega, \omega^{2}, 1), \qquad \mathcal{L} = k J = \mbox{diag}(l_{1},l_{2},l_{3}), \qquad \mathcal{Z} = k^{2} J^{2} = \mbox{diag}(z_{1},z_{2},z_{3}).
\end{align*}
Equation (\ref{3x3eq}) is the compatibility condition of the following Lax pair (see \cite{L3x3})
\begin{align}\label{Xlax}
\begin{cases}
X_x-[\mathcal{L},X]=\mathcal{U}X,\\
X_t-[\mathcal{Z},X]=\mathcal{V}X,
\end{cases}
\end{align}
where $k \in \C$  is the spectral parameter, $X(x,t,k)$ is a $3 \times 3$-matrix valued eigenfunction, and $\mathcal{U} = \mathcal{U}(x,t)$ and $\mathcal{V} = \mathcal{V}(x,t,k)$ are defined by
\begin{subequations} \label{expression for Ufrak}
\begin{align}
\mathcal{U} = &\; \begin{pmatrix} 0 & (1-\omega^2) \bar{q} & (1-\omega)q \\
(1-\omega)q & 0 & (1-\omega^2)\bar{q} \\
(1-\omega^2)\bar{q} & (1-\omega)q & 0 \end{pmatrix},
	\\ \nonumber
 \mathcal{V} = &\; k \begin{pmatrix} 
0 & (\omega^2-1) \bar{q} & (1-\omega^2)q \\
(\omega-1)q & 0 & (1-\omega)\bar{q} \\
(\omega-\omega^2)\bar{q} & (\omega^2-\omega)q & 0 \end{pmatrix}
	\\
& + (\bar{q}_x - 3q^2)\begin{pmatrix} 0 & \omega & 0 \\ 0 & 0 & \omega \\ \omega & 0 & 0 \end{pmatrix}
 + (q_x - 3\bar{q}^2)\begin{pmatrix} 0 & 0 & \omega^2 \\ \omega^2 & 0 & 0 \\ 0 & \omega^2 & 0 \end{pmatrix}. \end{align}
\end{subequations}
Indeed, a straightforward calculation shows that the compatibility condition of \eqref{Xlax}, i.e. $X_{xt}-X_{tx} = 0$, is equivalent to \eqref{3x3eq}. 
Assuming that $q(x,t)$ is a Schwartz class solution of (\ref{3x3eq}), we have
\begin{equation}
\lim_{x \to \pm \infty} \mathcal{U} = 0, \qquad \lim_{x \to \pm \infty} \mathcal{V} = 0.
\end{equation}
The Lax pair (\ref{Xlax}) possesses two convenient symmetries: If $F$ is any of the matrices $\mathcal{L}, \mathcal{Z}, \mathcal{U}, \mathcal{V}$, then
\begin{align}\label{Fsymm}
F(k) = \mathcal{A}F(\omega k) \mathcal{A}^{-1} = \mathcal{B}\overline{F(\overline{k})} \mathcal{B}, \qquad k \in \C,
\end{align}
where 
$$\mathcal{A} := \begin{pmatrix}
0 & 0 & 1 \\
1 & 0 & 0 \\
0 & 1 & 0
\end{pmatrix},
\qquad \mathcal{B} := \begin{pmatrix}
0 & 1 & 0 \\
1 & 0 & 0 \\
0 & 0 & 1
\end{pmatrix}.$$ 

\begin{remark}\upshape
The Lax pair (\ref{Xlax}) for equation \eqref{3x3eq} is similar to the Lax pair for the ``good'' Boussinesq equation analyzed in \cite{CL}. 
For example, the symmetries (\ref{Fsymm}) are the same as in \cite[Eqs. (3.18) and (3.19)]{CL}.
The main difference between (\ref{Xlax}) and the Lax pair analyzed in \cite{CL} is that whereas the matrices $\mathcal{U}$ and $\mathcal{V}$ in (\ref{Xlax}) are bounded at the origin, the analogs of $\mathcal{U}$ and $\mathcal{V}$ in \cite{CL} have double poles at the origin. There are also other, less essential, differences that we will indicate when necessary. Because of the evident similarities with \cite{CL}, we will omit some proofs in what follows.
\end{remark}

\subsection{The eigenfunctions $X$ and $Y$} 
In the remainder of this section, we will set $t=0$ and abuse notation by writing $\mathcal{U}(x)$ instead of $\mathcal{U}(x,0)$. We start by considering the $x$-part of \eqref{Xlax} evaluated at $t = 0$:
\begin{align}\label{xpart}
X_x - [\mathcal{L}, X] = \mathcal{U} X.
\end{align}
We define two $3 \times 3$-matrix valued solutions $X(x,k)$ and $Y(x,k)$ of (\ref{xpart}) as the solutions of the linear Volterra integral equations
\begin{subequations}\label{XYdef}
\begin{align}  \label{XYdefa}
 & X(x,k) = I - \int_x^{\infty} e^{(x-x')\widehat{\mathcal{L}(k)}} \big(\mathcal{U}(x')X(x',k) \big) dx',
  	\\ \label{XYdefb}
&  Y(x,k) = I + \int_{-\infty}^x e^{(x-x')\widehat{\mathcal{L}(k)}} \big(\mathcal{U}(x')Y(x',k)\big) dx'.
\end{align}
\end{subequations}
The six open subsets $\{D_n\}_1^6$ of Figure \ref{Gamma.pdf} can be written as
\begin{align*}
&D_1 = \{k \in \C\,|\, \re l_1 < \re l_2 < \re l_3\},
	& & D_2 = \{k \in \C\,|\, \re l_1 < \re l_3 < \re l_2\},
	\\
&D_3 = \{k \in \C\,|\, \re l_3 < \re l_1 < \re l_2\},
	& & D_4 = \{k \in \C\,|\, \re l_3 < \re l_2 < \re l_1\},
	\\
&D_5 = \{k \in \C\,|\, \re l_2 < \re l_3 < \re l_1\},
	& & D_6 = \{k \in \C\,|\, \re l_2 < \re l_1 < \re l_3\}.
\end{align*}
We let $\mathcal{S} := \{k \in \C \, |  \arg k \in (\frac{2\pi}{3}, \frac{4\pi}{3})\}$ denote the interior of $\bar{D}_3 \cup \bar{D}_4$. 

Proposition \ref{XYprop} below collects some basic properties of $X$ and $Y$. The statement is almost identical to \cite[Proposition 3.1]{CL} and we omit the proof; the only major difference is that in our case $X$ and $Y$ are well-defined at the origin.
\begin{proposition}\label{XYprop}
Suppose $q_0 \in \mathcal{S}(\R)$. 
Then the equations (\ref{XYdef}) uniquely define two $3 \times 3$-matrix valued solutions $X$ and $Y$ of (\ref{xpart}) with the following properties:
\begin{enumerate}[$(a)$]
\item The function $X(x, k)$ is defined for $x \in \R$ and $k \in (\omega^2 \bar{\mathcal{S}}, \omega \bar{\mathcal{S}}, \bar{\mathcal{S}})$, that is, the first column of $X(x,k)$ is defined for $x \in \R$ and $k \in \omega^2 \bar{\mathcal{S}}$, the second column of  $X(x,k)$ is defined for $x \in \R$ and $k \in \omega \bar{\mathcal{S}}$, etc. For each $k \in (\omega^2 \bar{\mathcal{S}}, \omega \bar{\mathcal{S}}, \bar{\mathcal{S}}) $, $X(\cdot, k)$ is smooth and satisfies (\ref{xpart}).

\item The function $Y(x, k)$ is defined for $x \in \R$ and $k \in (-\omega^2 \bar{\mathcal{S}}, -\omega \bar{\mathcal{S}}, -\bar{\mathcal{S}})$. For each $k \in (-\omega^2 \bar{\mathcal{S}}, -\omega \bar{\mathcal{S}}, -\bar{\mathcal{S}})$, $Y(\cdot, k)$ is smooth and satisfies (\ref{xpart}).

\item For each $x \in \R$, the function $X(x,\cdot)$ is continuous for $k \in (\omega^2 \bar{\mathcal{S}}, \omega \bar{\mathcal{S}}, \bar{\mathcal{S}})$ and analytic for $k \in (\omega^2 \mathcal{S}, \omega \mathcal{S}, \mathcal{S})$.

\item For each $x \in \R$, the function $Y(x,\cdot)$ is continuous for $k \in (-\omega^2 \bar{\mathcal{S}}, -\omega \bar{\mathcal{S}}, -\bar{\mathcal{S}})$ and analytic for $k \in (-\omega^2 \mathcal{S}, -\omega \mathcal{S}, -\mathcal{S})$.

\item For each $x \in \R$ and each $j = 1, 2, \dots$, the partial derivative $\frac{\partial^j X}{\partial k^j}(x, \cdot)$ has a continuous extension to $(\omega^2 \bar{\mathcal{S}}, \omega \bar{\mathcal{S}}, \bar{\mathcal{S}})$.

\item For each $x \in \R$ and each $j = 1, 2, \dots$, the partial derivative $\frac{\partial^j Y}{\partial k^j}(x, \cdot)$ has a continuous extension to $(-\omega^2 \bar{\mathcal{S}}, -\omega \bar{\mathcal{S}}, -\bar{\mathcal{S}})$.

\item For each $n \geq 1$, there are bounded smooth positive functions $f_+(x)$ and $f_-(x)$ of $x \in \R$ with rapid decay as $x \to +\infty$ and $x \to -\infty$, respectively, such that the following estimates hold for $x \in \R$ and $ j = 0, 1, \dots, n$:
\begin{subequations}\label{Xest}
\begin{align}\label{Xesta}
& \bigg|\frac{\partial^j}{\partial k^j}\big(X(x,k) - I\big) \bigg| \leq
f_+(x), \qquad k \in (\omega^2 \bar{\mathcal{S}}, \omega \bar{\mathcal{S}}, \bar{\mathcal{S}}),
	\\ \label{Xestb}
& \bigg|\frac{\partial^j}{\partial k^j}\big(Y(x,k) - I\big) \bigg| \leq
f_-(x), \qquad k \in (-\omega^2 \bar{\mathcal{S}}, -\omega \bar{\mathcal{S}}, -\bar{\mathcal{S}}).
\end{align}
\end{subequations}

\item $X$ and $Y$ obey the following symmetries for each $x \in \R$:
\begin{subequations}\label{XYsymm}
\begin{align}
&  X(x, k) = \mathcal{A} X(x,\omega k)\mathcal{A}^{-1} = \mathcal{B} \overline{X(x,\overline{k})}\mathcal{B}, 
\qquad k \in (\omega^2 \bar{\mathcal{S}}, \omega \bar{\mathcal{S}}, \bar{\mathcal{S}}),
	\\
&  Y(x, k) = \mathcal{A} Y(x,\omega k)\mathcal{A}^{-1} = \mathcal{B} \overline{Y(x,\overline{k})}\mathcal{B}, 
\qquad k \in (-\omega^2 \bar{\mathcal{S}}, -\omega \bar{\mathcal{S}}, -\bar{\mathcal{S}}).
\end{align}
\end{subequations}
\item If $q_0(x)$ has compact support, then, for each  $x \in \R$,  $X(x, k)$ and $Y(x, k)$ are defined and analytic for $k \in \C$ and $\det X = \det Y = 1$.
\end{enumerate}
\end{proposition}

\subsection{Asymptotics of $X$ and $Y$ as $k \to \infty$}
Following \cite[Section 3.3]{CL}, we first consider the following formal series
\begin{align*}
& X_{formal}(x,k) = I + \frac{X_1(x)}{k} + \frac{X_2(x)}{k^2} + \cdots,
	\\ \nonumber
& Y_{formal}(x,k) =  I + \frac{Y_1(x)}{k} + \frac{Y_2(x)}{k^2} + \cdots,
\end{align*}
normalized at $x = \infty$ and $x = -\infty$, respectively:
\begin{align}\label{Fjnormalization}
\lim_{x\to \infty} X_j(x) = \lim_{x\to -\infty} Y_j(x) = 0, \qquad j \geq 1.
\end{align}
Substituting $X_{formal}$ into (\ref{xpart}), we obtain the following recurrence relations for the coefficients $\{X_{j}\}_{j=1}^{+\infty}$ by identifying terms of the same order:
\begin{align}\label{xrecursive}
\begin{cases}
[J, X_{j+1}] = \partial_x X_{j}^{(o)} - (\mathcal{U}X_{j})^{(o)},
	\\
\partial_x X_{j}^{(d)} = (\mathcal{U}X_{j})^{(d)},
\end{cases}  \qquad j \geq 0,
\end{align}
where $X_{0} = I$, and $A^{(d)}$ and $A^{(o)}$ denote the diagonal and off-diagonal parts of a $3 \times 3$ matrix $A$, respectively. 
The first coefficients $X_{1}$ and $Y_{1}$ are given by
\begin{subequations}\label{X1Y1explicit}
\begin{align}
X_1(x) = & \; \begin{pmatrix}
0 & \omega \bar{q}_{0} & q_{0} \\
\omega^{2} q_{0} & 0 & \bar{q}_{0} \\
\omega^{2} \bar{q}_{0} & \omega q_{0} & 0
\end{pmatrix} - 3 \int_{\infty}^{x} |q_{0}|^{2} dx' \begin{pmatrix} \omega^2 & 0 & 0 \\ 
0 & \omega & 0 \\ 
0 & 0 & 1
\end{pmatrix},
	\\ 
Y_1(x) = & \begin{pmatrix}
0 & \omega \bar{q}_{0} & q_{0} \\
\omega^{2} q_{0} & 0 & \bar{q}_{0} \\
\omega^{2} \bar{q}_{0} & \omega q_{0} & 0
\end{pmatrix} - 3 \int_{-\infty}^{x} |q_{0}|^{2} dx' \begin{pmatrix} \omega^2 & 0 & 0 \\ 
0 & \omega & 0 \\ 
0 & 0 & 1
\end{pmatrix}.
\end{align}
\end{subequations}

Propositions \ref{XYprop2} shows that $X$ and $Y$ coincide with $X_{formal}$ and $Y_{formal}$ to all orders as $k \to \infty$. We omit the proof which is very similar to the proof of \cite[Proposition 3.2]{CL}. 

\begin{proposition}\label{XYprop2}
Suppose $q_0 \in \mathcal{S}(\R)$. 
As $k \to \infty$, $X$ and $Y$ coincide to all orders with $X_{formal}$ and $Y_{formal}$, respectively. More precisely, let $p \geq 0$ be an integer. Then the functions
\begin{align}\label{Xpdef}
&X_{(p)}(x,k) := I + \frac{X_1(x)}{k} + \cdots + \frac{X_{p}(x)}{k^{p}},
	\\ \nonumber
&Y_{(p)}(x,k) := I + \frac{Y_1(x)}{k} + \cdots + \frac{Y_{p}(x)}{k^{p}},
\end{align}
are well-defined and, for each integer $j \geq 0$,
\begin{subequations}\label{Xasymptotics}
\begin{align}\label{Xasymptoticsa}
& \bigg|\frac{\partial^j}{\partial k^j}\big(X - X_{(p)}\big) \bigg| \leq
\frac{f_+(x)}{|k|^{p+1}}, \qquad x \in \R, \  k \in (\omega^2 \bar{\mathcal{S}}, \omega \bar{\mathcal{S}}, \bar{\mathcal{S}}), \ |k| \geq 2,
	\\ \label{Xasymptoticsb}
& \bigg|\frac{\partial^j}{\partial k^j}\big(Y - Y_{(p)}\big) \bigg| \leq
\frac{f_-(x)}{|k|^{p+1}}, \qquad x \in \R, \  k \in (-\omega^2 \bar{\mathcal{S}}, -\omega \bar{\mathcal{S}}, -\bar{\mathcal{S}}), \ |k| \geq 2,
\end{align}
\end{subequations}
where $f_+(x)$ and $f_-(x)$ are bounded smooth positive functions of $x \in \R$ with rapid decay as $x \to +\infty$ and $x \to -\infty$, respectively.
\end{proposition}

\subsection{The scattering matrix}
The scattering matrix $s(k)$ was defined in (\ref{sdef}).
Our next proposition, whose proof is omitted, establishes several properties of $s(k)$. Properties $(a)$-$(e)$ and $(g)$ are very similar (though not identical) to their counterparts in \cite[Proposition 3.5]{CL}. The behavior of $s(k)$ as $k \to 0$ is stated in item $(f)$; this item differs significantly from the corresponding item in \cite[Proposition 3.5]{CL} but can be easily proved using Proposition \ref{XYprop}.

\begin{proposition}\label{sprop}
Let $\R_+ = (0,+\infty)$ and $\R_- = (-\infty, 0)$. Suppose $q_0 \in \mathcal{S}(\R)$. Then the scattering matrix $s(k)$ defined in (\ref{sdef}) has the following properties:
\begin{enumerate}[$(a)$]
\item The entries of $s(k)$ are defined and continuous for $k$ in
\begin{align}\label{sdomainofdefinition}
 \begin{pmatrix}
 \omega^2 \bar{\mathcal{S}} & \R_+ & \omega \R_+ \\
 \R_+ & \omega \bar{\mathcal{S}} & \omega^2 \R_+ \\
 \omega \R_+ & \omega^2 \R_+ & \bar{\mathcal{S}}
 \end{pmatrix},
\end{align}
that is, the $(11)$ entry of $s(k)$ is defined and continuous for $k \in \omega^2 \bar{\mathcal{S}}$, etc. 
 
\item The diagonal entries of $s(k)$  are analytic in the interior of their domains of definition as given in (\ref{sdomainofdefinition}). 
 
\item For $j = 1, 2, \dots$, the derivative $\partial_k^js(k)$ is well-defined and continuous for $k$ in (\ref{sdomainofdefinition}).

\item $s(k)$ obeys the symmetries
\begin{align}\label{symmetry of s}
&  s(k) = \mathcal{A} s(\omega k)\mathcal{A}^{-1} = \mathcal{B} \overline{s(\overline{k})}\mathcal{B}.
\end{align}

\item $s(k)$ approaches the identity matrix as $k \to \infty$. More precisely, there are diagonal matrices $\{s_j\}_1^\infty$ such that
\begin{align*}\nonumber
\Big|\partial_k^j \Big(s(k) - I - \sum_{j=1}^N \frac{s_j}{k^j}\Big)\Big| = O(k^{-N-1}), \qquad k \to \infty,
\ k \in \begin{pmatrix}
 \omega^2 \bar{\mathcal{S}} & \R_+ & \omega \R_+ \\
 \R_+ & \omega \bar{\mathcal{S}} & \omega^2 \R_+ \\
 \omega \R_+ & \omega^2 \R_+ & \bar{\mathcal{S}}
 \end{pmatrix},
\end{align*}
for $j = 0, 1, \dots, N$ and each integer $N \geq 1$. In particular, the off-diagonal entries of $s(k)$ have rapid decay as $k \to \infty$.

\item We have
\begin{align}\label{s at 0}
s(k) = s^{(0)} + s^{(1)}k + \ldots, \qquad k \to 0, \; k \in \begin{pmatrix}
 \omega^2 \bar{\mathcal{S}} & \R_+ & \omega \R_+ \\
 \R_+ & \omega \bar{\mathcal{S}} & \omega^2 \R_+ \\
 \omega \R_+ & \omega^2 \R_+ & \bar{\mathcal{S}}
 \end{pmatrix},
\end{align}
and the expansion can be differentiated termwise any number of times.
\item If $q_{0}(x)$ has compact support, then $s(k)$ is defined and analytic for $k \in \C$, $\det s = 1$ for $k \in \C$, and
\begin{align}\label{XYs} 
X(x,k) = Y(x,k)e^{x\widehat{\mathcal{L}(k)}} s(k), \qquad k \in \C.
\end{align}
\end{enumerate}
\end{proposition}

\subsection{Analysis of $X^A$, $Y^A$ and $s^{A}$} Since the Lax pair associated to \eqref{3x3eq} is third order, the columns of $X$ and $Y$ alone are not sufficient to define a RH problem. Following \cite{L3x3}, we need to consider $X^{A}$ and $Y^{A}$ defined as the solutions of the following Volterra equations:
\begin{subequations}\label{XAYAdef}
\begin{align}
& X^{A}(x,k) = I + \int_{x}^{\infty}e^{-(x-x')\widehat{\mathcal{L}(k)}} \big(\mathcal{U}(x')^{T}X^{A}(x',k)\big)dx',  \\
& Y^{A}(x,k) = I - \int_{-\infty}^{x}e^{-(x-x')\widehat{\mathcal{L}(k)}} \big(\mathcal{U}(x')^{T}Y^{A}(x',k)\big)dx', \end{align}
\end{subequations}
and the associated spectral function $s^{A}(k)$ defined in (\ref{sAdef}). If $q_{0}$ is compactly supported, all entries of $X,X^{A},Y,Y^{A},s,s^{A}$ are well defined for all $k \in \mathbb{C}$, and $X^{A}$, $Y^{A}$, and $s^{A}$ are simply the inverse transpose of $X$, $Y$ and $s$, respectively. For general $q_{0} \in \mathcal{S}(\mathbb{R})$, the properties of $X^{A}$, $Y^{A}$, and $s^{A}$ can be established by adapting \cite[Sections 3.6-3.8]{CL} in a similar way as was done above for $X$, $Y$, and $s$. We omit this analysis. 

\subsection{Proof of Lemma \ref{r1r2lemma}}\label{r1r2subsec}
Suppose $q_0 \in \mathcal{S}(\R)$ is such that Assumption \ref{solitonlessassumption} holds. 
Recall from (\ref{3x3r1r2def}) that $r_1 = s_{12}/s_{11}$. Statements $(a)$ and $(c)$ of Proposition \ref{sprop} imply that $s_{12}(k)$ and $s_{11}(k)$ are smooth for $k \in (0,\infty)$. Since $s_{11}$ has no zeros by Assumption \ref{solitonlessassumption}, it follows that $r_1 \in C^\infty((0,\infty))$. Moreover, statement $(e)$ of Proposition \ref{sprop} implies that $r_1$ satisfies (\ref{r1rapiddecay}), and statement $(f)$ implies that $r_{1}$ and its derivatives have continuous boundary values at $k=0$, showing that $r_1$  also satisfies (\ref{r1r2atzero}). This completes the proof of parts (i), (ii) and (iii) of Lemma \ref{r1r2lemma} for $r_{1}$; the proof for $r_{2}$ follows similarly from the properties of $X^{A}$, $Y^{A}$, and $s^{A}$. 

It remains to prove Lemma \ref{r1r2lemma} (iv), i.e. that 
\begin{align*}
|r_{1}(k)|<1 \quad \mbox{for all } k \geq 0 \qquad \mbox{ and } \qquad |r_{2}(k)| <1 \quad \mbox{for all }  k \leq 0.
\end{align*}
For $k\leq 0$, all four entries $\{s_{ij}^{A}(k)\}_{i,j=1}^{2}$ are well-defined, so by uniqueness of $s$ we conclude that $s_{11}^{A}(k)s_{22}^{A}(k)-s_{12}^{A}(k)s_{21}^{A}(k) = s_{33}(k)$. Hence, using also the symmetries $s^{A}(k) = \mathcal{B}\overline{s^{A}(\overline{k})}\mathcal{B}$ and \eqref{symmetry of s}, we find
\begin{align}\label{lol2}
1-|r_{2}(k)|^{2} = 1 - \bigg| \frac{s_{12}^{A}(k)}{s_{11}^{A}(k)} \bigg|^{2} = \frac{s_{11}^{A}(k)s_{22}^{A}(k)-s_{12}^{A}(k)s_{21}^{A}(k)}{|s_{11}^{A}(k)|^{2}} = \frac{s_{33}(k)}{|s_{11}^{A}(k)|^{2}} = \frac{s_{11}(\omega^{2}k)}{|s_{11}^{A}(k)|^{2}}.
\end{align}
Since the right-hand side of \eqref{lol2} is non-zero for all $k \leq 0$ by Assumption \ref{solitonlessassumption}, and since the left-hand side is real and tends to $1$ as $k \to -\infty$ by Lemma \ref{r1r2lemma} (iii), we conclude that $1-|r_{2}(k)|^{2}>0$ for all $k \leq 0$. The proof of $|r_{1}(k)| <1$ for $k \geq 0$ is similar and we omit it.

\section{Proof of Theorem \ref{RHth}}\label{proofsec}
Let $q(x,t)$ be a Schwartz class solution of (\ref{3x3eq}) with existence time $T \in (0, \infty]$ and initial data $q_0 \in \mathcal{S}(\R)$. Suppose Assumption \ref{solitonlessassumption} holds and define $r_1(k)$ and $r_2(k)$ in terms of $q_0$ by (\ref{3x3r1r2def}).

\subsection{Proof of $(i)$} 
We will construct the solution $m$ of RH problem \ref{RH problem for m} explicitly in terms of solutions of certain Fredholm equations. 
Let $m_{n}$ denote the restriction of $m$ to $D_{n}$. For each $n = 1, \dots, 6$, define $m_n(x,k)$ for $k \in D_n$ as the solution of the following system of Fredholm integral equations: 
\begin{align}\label{Mndef}
(m_n)_{ij}(x,k) = \delta_{ij} + \int_{\gamma_{ij}^n} \left(e^{(x-x')\widehat{\mathcal{L}(k)}} (\mathcal{U}(x')m_n(x',k)) \right)_{ij} dx', \qquad  i,j = 1, 2,3,
\end{align}
where the contours $\gamma^n_{ij}$, $n = 1, \dots, 6$, $i,j = 1,2,3$, are defined by
 \begin{align} \label{gammaijndef}
 \gamma_{ij}^n =  \begin{cases}
 (-\infty,x),  & \re  l_i(k) < \re  l_j(k), 
	\\
(+\infty,x),  \quad & \re  l_i(k) \geq \re  l_j(k),
\end{cases} \quad \text{for} \quad k \in D_n.
\end{align}
Let $\mathsf{Z}$ denote the set of zeros of the Fredholm determinants associated with (\ref{Mndef}). The next proposition establishes the basic properties of $m_{n}$. 

\begin{proposition}\label{Mnprop}
Suppose $q_0 \in \mathcal{S}(\R)$. 
Then the equations (\ref{Mndef}) uniquely define six $3 \times 3$-matrix valued solutions $\{m_n\}_1^6$ of (\ref{xpart}) with the following properties:
\begin{enumerate}[$(a)$]
\item The function $m_n(x, k)$ is defined for $x \in \R$ and $k \in \bar{D}_n \setminus \mathsf{Z}$. For each $k \in \bar{D}_n  \setminus \mathsf{Z}$, $m_n(\cdot, k)$ is smooth and satisfies (\ref{xpart}).

\item For each $x \in \R$, the function $m_n(x,\cdot)$ is continuous for $k \in \bar{D}_n \setminus \mathsf{Z}$ and analytic for $k \in D_n\setminus \mathsf{Z}$.

\item For each $\epsilon > 0$, there exists a $C = C(\epsilon)$ such that
\begin{align}\label{Mnbounded}
|m_n(x,k)| \leq C, \qquad x \in \R, \ k \in \bar{D}_n, \ \dist(k, \mathsf{Z}) \geq \epsilon.
\end{align}

\item For each $x \in \R$ and each $j = 1, 2, \dots$, the partial derivative $\frac{\partial^j m_n}{\partial k^j}(x, \cdot)$ has a continuous extension to $\bar{D}_n \setminus \mathsf{Z}$.

\item $\det m_n(x,k) = 1$ for $x \in \R$ and $k \in \bar{D}_n \setminus \mathsf{Z}$.

\item For each $x \in \R$, the sectionally analytic function $m(x,k)$ defined by $m(x,k) = m_n(x,k)$ for $k \in D_n$ satisfies the symmetries
\begin{align}\label{Msymm}
 m(x, k) = \mathcal{A} m(x,\omega k)\mathcal{A}^{-1} = \mathcal{B} \overline{m(x,\overline{k})}\mathcal{B}, \qquad k \in \C \setminus \mathsf{Z}.
\end{align}
\end{enumerate}
\end{proposition}
\begin{proof}
We refer to \cite[Proposition 4.1]{CL} for a detailed proof in a similar, though more complicated, situation.
\end{proof}

Now we turn to the behavior of $m$ as $k \to \infty$. Since $\gamma_{ii}^n = (\infty,x)$ for $i = 1,2,3$ and $n = 1, \dots,6$, it follows that $m$ has the same asymptotic behavior at $\infty$ as $X$. The proof of the next lemma consists of minor adaptations of \cite[Lemma 4.3]{CL}.

\begin{lemma}\label{minftylemma}
Suppose $q_0 \in \mathcal{S}(\R)$ and $q_0 \not\equiv 0$. Given an integer $p \geq 1$, let $X_{(p)}(x,k)$ be the function defined in (\ref{Xpdef}). Then there exists an $R > 0$ such that
\begin{align}
& \big|m(x,k) - X_{(p)}(x,k) \big| \leq
\frac{C}{|k|^{p+1}}, \qquad x \in \R, \  k \in \C \setminus \Gamma, \ |k| \geq R.
\end{align}
\end{lemma}

Following \cite[Lemmas 4.4, 4.5 and 4.6]{CL}, we conclude that $m$ satisfies the jumps of the RH problem \ref{RH problem for m}, except at the points $k \in \mathsf{Z}$.
\begin{lemma}[Jump condition for $m$]\label{Mjumplemma}
Let $q_0 \in \mathcal{S}(\R)$.
For each $x \in \R$, $m(x,k)$ satisfies the jump condition
\begin{align*}
m_+(x,k) = m_-(x, k) v(x, 0, k), \qquad k \in \Gamma \setminus \mathsf{Z},
\end{align*}
where $v$ is the jump matrix defined in (\ref{vdef}).
\end{lemma}

We have not yet proved that $m$ satisfies the RH problem \ref{RH problem for m}, since $m$ is only defined for $\C\setminus \mathsf{Z}$. The next lemma shows that $m$ can be explicitly written in terms of $X$, $Y$, $s$, $X^{A}$, $Y^{A}$, and $s^{A}$. Using the properties of these functions, Lemma \ref{M1XYlemma} allows us to show that $m$ can be analytically continued from $\C\setminus \mathsf{Z}$ to $\C$. The proof is identical to the proof of \cite[Lemma 4.7]{CL}, so we omit it.

\begin{lemma}\label{M1XYlemma}
Let $q_0 \in \mathcal{S}(\R)$.
The function $m_1$ can be expressed in terms of the entries of $X,Y,X^A, Y^A, s$, and $s^A$ as follows:
\begin{align}\label{m1 in terms of X Y and s}
m_1 = \begin{pmatrix} 
X_{11} & \frac{Y_{31}^AX_{23}^A - Y_{21}^AX_{33}^A}{s_{11}} & \frac{Y_{13}}{s_{33}^A} \\
X_{21} & \frac{Y_{11}^AX_{33}^A - Y_{31}^AX_{13}^A}{s_{11}} & \frac{Y_{23}}{s_{33}^A} \\
X_{31} & \frac{Y_{21}^AX_{13}^A - Y_{11}^AX_{23}^A}{s_{11}} & \frac{Y_{33}}{s_{33}^A} 
\end{pmatrix},
\end{align}
for all $x \in \R$ and $k \in \bar{D}_1 \setminus \mathsf{Z}$. There exist similar expressions for $m_{2},\ldots,m_{6}$.
\end{lemma}

Since $q_0$ satisfies Assumption \ref{solitonlessassumption}, it follows from \eqref{m1 in terms of X Y and s} that $m$ is bounded for $k \in \bar{D}_{1}\setminus \mathsf{Z}$. Therefore, we can define the value of $m(x,t,k)$ at a point $k_j \in \mathsf{Z} \cap \bar{D}_n$ by continuity:
\begin{align}\label{Mnkjdef}
m_n(x,k_j) = \lim_{\substack{ k\to k_j \\ k \in \bar{D}_n \setminus \mathsf{Z}}} m_n(x,k).
\end{align}
We conclude that $\mathsf{Z}$ can be replaced with the empty set in all the above results.

The next lemma establishes the asymptotics of $m$ as $k \to 0$. The proof is similar to (but simpler than) the proof of \cite[Lemma 4.9]{CL}.

\begin{lemma}\label{Mat1lemma}
Suppose $q_0 \in \mathcal{S}(\R)$ is such that Assumption \ref{solitonlessassumption} holds.
Let $p \geq 1$ be an integer.
Then there are $3 \times 3$-matrix valued functions $\{\mathfrak{m}_n^{(l)}(x)\}$, $n = 1,\dots,6$, $l = 0, \dots, p$, with the following properties:
\begin{enumerate}[$(a)$]
\item The function $m$ satisfies, for $x \in \R$,
\begin{align*}
& \bigg|m_n(x,k) - \sum_{l=0}^p \mathfrak{m}_n^{(l)}(x)k^l\bigg| C\leq
|k|^{p+1}, \qquad |k| \leq \frac{1}{2}, \ k \in \bar{D}_n.
\end{align*}
\item For each $n$ and each $l \geq 0$, $\mathfrak{m}_n^{(l)}(x)$ is a smooth function of $x \in \R$.
\end{enumerate}
\end{lemma}
It follows from the above lemmas that $m(x,k)$ (or more precisely, $m(x,0,k)$) defined by \eqref{Mndef} satisfies the RH problem \ref{RH problem for m} with $t=0$. Moreover, using Lemma \ref{minftylemma} and recalling equations (\ref{X1Y1explicit}) and (\ref{Xpdef}), we infer that (\ref{qlim}) holds for $t = 0$. These arguments can be extended to any $t \in [0,T)$ as in \cite[Section 5]{CL} by replacing $u_0(x)$  with $u(x,t)$ in the above steps. Since the jump matrix has unit determinant, uniqueness of the RH problem for $m$ follows in the standard way. 
This completes the proof of part $(i)$ of Theorem \ref{RHth}.

\subsection{Proof of $(ii)$} 
The RH problem \ref{RH problem for m} for $m$ is a regularized version of the RH problem associated with the ``good'' Boussinesq equation obtained in \cite{CL}, and the asymptotic analysis of this regularized RH problem was performed as a step in the derivation of long-time asymptotics for the ``good'' Boussinesq equation presented in \cite{CLW}.
More precisely, for $\zeta:=x/t \in (\zeta_0,\infty)$, it follows from \cite{CLW} that 
\begin{align*}
\lim_{k\to \infty} \big(k \, m(x,t,k)\big)_{13} 
& =\lim_{k\to \infty} \big(k \, \hat{m}(x,t,k) \Delta^{-1}(\zeta,k)\big)_{13} 
	\\
& =\lim_{k\to \infty} \big(k (\hat{m}(x,t,k) - I) \Delta^{-1}(\zeta,k)\big)_{13} + \lim_{k\to \infty} \big(k \Delta^{-1}(\zeta,k)\big)_{13},
\end{align*}
where $\Delta^{-1}$ is given by
\begin{align*}
\Delta^{-1}(\zeta,k) = \begin{pmatrix} 
\frac{\delta_3(\zeta, k)}{\delta_1(\zeta, k)} & 0 & 0 \\
0 & \frac{\delta_1(\zeta, k)}{\delta_5(\zeta, k)} & 0 \\
0 & 0 & \frac{\delta_5(\zeta, k)}{\delta_3(\zeta, k)} \end{pmatrix}, \quad k \in \mathbb{C}\setminus \big( [k_{0},+\infty) \cup \omega [k_{0},+\infty) \cup \omega^{2} [k_{0},+\infty) \big)
\end{align*}
with $k_0 = k_0(\zeta) = \zeta/2$ and
\begin{align}
& \delta_1(\zeta, k) = e^{-i  \nu \ln_{0}(k-k_{0})}e^{-\chi_{1}(\zeta,k)}, & & k \in \C \setminus [k_{0},+\infty), \label{def of delta1} \\
& \delta_3(\zeta,k) = \delta_1(\zeta,\omega^{2}k), & & k \in \mathbb{C}\setminus \omega [k_{0},+\infty), \label{def of delta3} \\
& \delta_5(\zeta,k) = \delta_1(\zeta,\omega k), & & k \in \mathbb{C}\setminus \omega^{2} [k_{0},+\infty). \label{def of delta5}
\end{align}
The quantities $\nu = \nu(\zeta) \geq 0$ and $\chi_{1}(\zeta,k)$ that appear in the definition of $\delta_{1}$ are given by
\begin{align*}
\nu = - \frac{1}{2\pi}\ln(1-|r_1(k_{0})|^{2}), \qquad \chi_{1}(\zeta,k) = \frac{1}{2\pi i} \int_{k_{0}}^{\infty}  \ln_{0}(k-s) d\ln(1-|r_1(s)|^{2}),
\end{align*}
and $\ln_0(k) = \ln|k| + i \arg_0 k$ with $\arg_0 k\in [0,2\pi)$. Since $\Delta^{-1}$ is a diagonal matrix, clearly
\begin{align*}
\lim_{k\to \infty} \big(k \Delta^{-1}(\zeta,k)\big)_{13} = 0.
\end{align*}
As $t \to \infty$, we have, by \cite[Section 6]{CLW},
\begin{align}\label{limlhatm2}
\lim_{k\to \infty} k(\hat{m}(x,t,k) - I)
= & \; \frac{ \sum_{j=0}^2 \omega^j \mathcal{A}^{-j} Y(\zeta,t) m_1^{X}(\mathrm{q}(\zeta)) Y(\zeta,t)^{-1}\mathcal{A}^j}{3^{1/4}\sqrt{2}\sqrt{t}} + O(t^{-1}\ln t)
\end{align}
uniformly for $\zeta := x/t \in \mathcal{I}$ where $\mathcal{I}$ is any compact subset of $(\zeta_0,\infty)$. The explicit expressions for $Y(\zeta,t)$, $\mathrm{q}(\zeta)$ and $m_{1}^{X}$ are recalled below. Note that $\Delta^{-1}(\zeta, k) = I + O(k^{-1})$ as $k \to \infty$. Hence, as $t \to \infty$
\begin{align}\nonumber
q(x,t) & = \lim_{k\to \infty} \big(k \, m(x,t,k)\big)_{13} 
 = \lim_{k\to \infty} \big(k (\hat{m}(x,t,k) - I) \Delta^{-1}(\zeta, k)\big)_{13} 
  = \lim_{k\to \infty} \big(k \hat{m}(x,t,k)\big)_{13} 
	\\ \label{qasymptotics}
& = \frac{ \sum_{j=0}^2 \omega^j \big(\mathcal{A}^{-j} Y(\zeta,t) m_1^{X}(q(\zeta)) Y(\zeta,t)^{-1}\mathcal{A}^j\big)_{13}}{3^{1/4}\sqrt{2}\sqrt{t}}+ O(t^{-1}\ln t)
\end{align}
uniformly for $\zeta\in \mathcal{I}$.
Using that $\mathrm{q} = \mathrm{q}(\zeta) = r_1(k_0) \in \D:= \{z \in \mathbb{C}: |z|<1\}$ and that
\begin{align}
& Y(\zeta,t) = \begin{pmatrix} 
d_0^{1/2}(\zeta, t)e^{-\frac{t}{2}\Phi_{21}(\zeta, k_0)} & 0 & 0 \\
0 & d_0^{-1/2}(\zeta, t)e^{\frac{t}{2}\Phi_{21}(\zeta, k_0)} & 0 \\
0 & 0 & 1 \end{pmatrix}, \nonumber \\
& \Phi_{21}(\zeta,k_0) =  \omega(\omega-1) k_0^2 = -i\sqrt{3} k_0^2, \label{def of Phi21 at k0} \\
& d_0(\zeta, t) = (2\sqrt{3}t)^{-i\nu} e^{2\chi_{1}(\zeta,k_{0})} \delta_3(\zeta,k_{0})\delta_5(\zeta,k_{0}), \label{def of d0} \\
& m_1^X(\mathrm{q}) =\begin{pmatrix} 
0 & \beta_{12} & 0 	\\
\beta_{21} & 0 & 0 \\
0 & 0 & 0 \end{pmatrix}, \nonumber
\end{align}
where $\beta_{12}$ and $\beta_{21}$ are defined by
\begin{align*}
\beta_{12}=\frac{\sqrt{2\pi}e^{-\frac{\pi i}{4}} e^{-\frac{5\pi \nu}{2}} }{\bar{\mathrm{q}} \Gamma(-i\nu)} ,\qquad
\beta_{21}=\frac{\sqrt{2\pi}e^{\frac{\pi i}{4}} e^{\frac{3\pi \nu}{2}} }{\mathrm{q} \Gamma(i\nu)},
\end{align*}
we obtain after a straightforward calculation that
$$\big(\mathcal{A}^{-j} Y(\zeta,t) m_1^{X}(\mathrm{q}(\zeta)) Y(\zeta,t)^{-1}\mathcal{A}^j\big)_{13}
= \begin{cases} 0, & j = 0,2, \\
\frac{\beta_{21} e^{t\Phi_{21}(\zeta, k_0)}}{d_0(\zeta,t)}, & j = 1.
\end{cases}$$
Hence (\ref{qasymptotics}) yields
\begin{align}\label{q asymp1}
q(x,t) & = \frac{\omega \beta_{21} e^{t\Phi_{21}(\zeta, k_0)}}{3^{1/4}d_0(\zeta, t) \sqrt{2t}}+ O(t^{-1}\ln t)
= \frac{\omega e^{t\Phi_{21}(\zeta, k_0)}}{3^{1/4}d_0(\zeta, t) \sqrt{2t}}\frac{\sqrt{2\pi}e^{\frac{\pi i}{4}} e^{\frac{3\pi \nu}{2}} }{\mathrm{q}(\zeta)  \Gamma(i\nu)} + O(t^{-1}\ln t)
\end{align}
as $t \to \infty$ uniformly for $\zeta\in \mathcal{I}$. Substituting \eqref{def of delta1}-\eqref{def of delta5} into  \eqref{def of d0}, we get
\begin{align*}
& d_0(\zeta, t) = \; (2\sqrt{3}t)^{-i\nu}  e^{2\chi_{1}(\zeta,k_{0})} \delta_1(\zeta, \omega^2 k_0) \delta_1(\zeta, \omega k_0)
	\\
& =  (2\sqrt{3}t)^{-i\nu}  \exp \bigg( -i  \nu \ln_{0}(\omega^2 k_0 -k_{0})-i  \nu \ln_{0}(\omega k_0 -k_{0})  \\
& \hspace{3.2cm}+ 2\chi_{1}(\zeta,k_{0}) -\chi_{1}(\zeta, \omega^2 k_0) -\chi_{1}(\zeta, \omega k_0) \bigg).
\end{align*}
Hence, using that
\begin{align*}
& \ln_{0}(\omega^2 k_0 -k_{0}) = \ln|\omega^{2} k_{0}-k_{0}| + \frac{7 \pi i}{6}, \qquad \ln_{0}(\omega k_{0}-k_{0}) = \ln|\omega k_{0}-k_{0}|+ \frac{5\pi i}{6}, \\
& 2\ln_{0}(k_0-s) - \ln_{0}(\omega^2 k_0-s) - \ln_{0}(\omega k_0-s) = 2 \ln \bigg| \frac{s-k_{0}}{s-\omega k_{0}} \bigg|,
\end{align*}
we obtain
\begin{align*}
d_{0}(\zeta,t) = \exp \bigg( -i\nu \ln(2\sqrt{3}t)-i  \nu \ln(3k_0^2) + 2\pi \nu +\frac{1}{\pi i} \int_{k_{0}}^{+\infty} \ln\bigg|\frac{s-k_0}{s - \omega k_0}\bigg| d\ln(1-|r_1(s)|^{2}) \bigg).
\end{align*}
Substituting this expression into \eqref{q asymp1}, and using \eqref{def of Phi21 at k0} and $\mathrm{q}=r_{1}(k_{0})$, we arrive at
\begin{align}\label{asymp for q final comp}
q(x,t) = &\; \frac{\sqrt{\pi} e^{-\frac{\pi \nu}{2}}}{3^{1/4} \sqrt{t}} 
\frac{e^{i\phi}}{r_1(k_0) \Gamma(i\nu)}
 + O(t^{-1}\ln t), \qquad t \to \infty,
\end{align}
uniformly for $\zeta\in \mathcal{I}$,
where
\begin{align*}
\phi = \phi(\zeta) = \frac{11\pi}{12}   + \nu \ln(6\sqrt{3}tk_0^2) - \sqrt{3} k_0^2 t
+ \frac{1}{\pi} \int_{k_{0}}^{+\infty} \ln\bigg|\frac{s-k_0}{s - \omega k_0}\bigg| d\ln(1-|r_1(s)|^{2}).
\end{align*}
This completes the proof of part (ii) of Theorem \ref{RHth}.

\subsection{Proof of $(iii)$.} \label{proofofiii}
Let us indicate the dependence on the initial data $q_{0}\in \mathcal{S}(\mathbb{R})$ explicitly by writing $X(x,k;q_{0})$ for $X(x,k)$ etc. Let $q(x,t)$ be a Schwartz class solution of \eqref{3x3eq} with existence time $T \in (0,\infty]$ and initial data $q_{0}$, and define $f(x,t)=-q(-x,t)$ and $f_{0}(x)=-q_{0}(-x)$. It is easy to check that $f(x,t)$ is a Schwartz class solution of \eqref{3x3eq} with existence time $T \in (0,\infty]$ and initial data $f_{0} \in \mathcal{S}(\mathbb{R})$. Combining this symmetry with \eqref{asymp for q final comp}, we infer that
\begin{align}\label{asymp for -q(-x) final comp}
q(x,t) = &\; -\frac{\sqrt{\pi} e^{-\frac{\pi \nu(-\zeta;f_{0})}{2}}}{3^{1/4} \sqrt{t}} 
\frac{e^{i\phi(-\zeta;f_{0})}}{r_1(-k_{0};f_{0}) \Gamma(i\nu(-\zeta;f_{0}))}
 + O(t^{-1}\ln t), \qquad t \to \infty,
\end{align}
uniformly for $x/t$ in compact subsets of $(-\infty,0)$. It follows from \eqref{mathsfUdef intro} that
$\mathcal{U}(-x;f_{0}) = - \mathcal{U}(x;q_{0})$. Since $\mathcal{L}(-k)=-\mathcal{L}(k)$, we deduce from \eqref{XYdef} and (\ref{XAYAdef}) that
\begin{align*}
X(-x,-k;f_{0}) = Y(x,k;q_{0}), \qquad X^{A}(-x,-k;f_{0}) = Y^{A}(x,k;q_{0}).
\end{align*}
Hence, by \eqref{XYs}, for all compactly supported $q_{0}\in \mathcal{S}(\mathbb{R})$,
\begin{align*}
s(-k;f_{0}) = (s^{A}(k;q_{0}))^{T},  \qquad s^{A}(-k;f_{0}) = (s(k;q_{0}))^{T}.
\end{align*}
In particular, for $k\leq 0$, we have
\begin{align*}
r_{1}(-k;f_{0}) = \frac{s(-k;f_{0})_{12}}{s(-k;f_{0})_{11}} = \frac{s^{A}(k;q_{0})_{21}}{s^{A}(k;q_{0})_{11}} =: \tilde{r}_{2}(k;q_{0}).
\end{align*}
Furthermore, the symmetry $s^{A}(k) = \mathcal{B}\overline{s^{A}(\overline{k})}\mathcal{B}$ implies that
\begin{align*}
|\tilde{r}_{2}(k;q_{0})| = |r_{2}(k;q_{0})|, \qquad k \leq 0, 
\end{align*}
and therefore 
\begin{align*}
\nu(-\zeta;f_{0}) & = - \frac{1}{2\pi} \ln (1-|r_{1}(-k_{0};f_{0})|^{2}) = - \frac{1}{2\pi} \ln (1-|r_{2}(k_{0};q_{0})|^{2}) =: \tilde{\nu}, \\
\phi(-\zeta;f_{0}) & = \frac{11\pi}{12}   + \nu(-\zeta;f_{0}) \ln(6\sqrt{3}tk_0^2) - \sqrt{3} k_0^2 t
+ \frac{1}{\pi} \int_{-k_{0}}^{+\infty} \ln\bigg|\frac{s+k_0}{s + \omega k_0}\bigg| d\ln(1-|r_1(s;f_{0})|^{2}) \\
& = \frac{11\pi}{12}   + \tilde{\nu} \ln(6\sqrt{3}tk_0^2) - \sqrt{3} k_0^2 t
+ \frac{1}{\pi} \int_{-\infty}^{k_{0}} \ln\bigg|\frac{s-k_0}{s - \omega k_0}\bigg| d\ln(1-|r_2(s;q_{0})|^{2}).
\end{align*}
This concludes the proof of part (iii) of Theorem \ref{RHth}.

\appendix

\section{Upper bounds on $|r_{1}|$ and $|r_{2}|$}\label{section:appendix}
We will prove that if $q_{0} \in \mathcal{S}(\mathbb{R})$ has compact support, then
\begin{align}\label{general upper bounds}
|r_{1}(k;q_{0})|<1 \quad \mbox{for all } k \geq 0 \qquad \mbox{ and } \qquad |r_{2}(k;q_{0})| <1 \quad \mbox{for all }  k \leq 0,
\end{align}
where we have indicated the dependence of $r_1$ and $r_2$ on the initial data $q_0$ for later convenience. We will also prove that
\begin{align}\label{general upper bounds unbounded}
|r_{1}(k;q_{0})|\leq 1 \quad \mbox{for all } k \geq 0 \qquad \mbox{ and } \qquad |r_{2}(k;q_{0})| \leq1 \quad \mbox{for all }  k \leq 0
\end{align}
holds whenever $q_{0} \in \mathcal{S}(\mathbb{R})$. Note from \eqref{ljzjdef} and \eqref{mathsfUdef intro} that 
\begin{align*}
\mathcal{L}(k) = -\mathcal{L}(-k), \qquad \mathcal{U}(x;q_{0}) = - \mathcal{U}^{T}(x;\omega^{2}\bar{q}_{0}).
\end{align*}
This implies by \eqref{XXAdef intro}, \eqref{sdef}, and \eqref{sAdef} that 
\begin{align*}
X(x,k;q_{0}) = X^{A}(x,-k;\omega^{2}\bar{q}_{0}) \quad \mbox{ and } \quad s(k;q_{0}) = s^{A}(-k;\omega^{2}\bar{q}_{0}),
\end{align*}
which, by \eqref{3x3r1r2def}, in turn implies that $r_{1}(k;q_{0}) = r_{2}(-k;\omega^{2}\bar{q}_{0})$. Hence, to prove \eqref{general upper bounds} and \eqref{general upper bounds unbounded}, it is sufficient to prove that 
\begin{align}
& |r_{2}(k;q_{0})| <1 \qquad \mbox{for all }  k \leq 0  \mbox{ and for all compactly supported } q_{0} \in \mathcal{S}(\mathbb{R}), \label{lol3} \\
& |r_{2}(k;q_{0})| \leq 1 \qquad \mbox{for all }  k \leq 0  \mbox{ and for all } q_{0} \in \mathcal{S}(\mathbb{R}). \label{lol4}
\end{align}

Let us fix $q_{0} \in \mathcal{S}(\mathbb{R})$ and $k \in (-\infty,0]$, and assume that $q_{0}$ is compactly supported. We will prove that $|r_{2}(k;q_{0})| <1$ for $k \leq 0$. Let $x_{l}\in \mathbb{R}$ be a point to the left of the support of $q_{0}$. By \eqref{XYdefb}, we have $Y(x_{l},k)=I$. Hence, by taking the inverse transpose of \eqref{symmetry of s} and \eqref{XYs}, we see that
\begin{align*}
1-|r_{2}(k)|^{2} = 1 - \bigg| \frac{s_{12}^{A}(k)}{s_{11}^{A}(k)} \bigg|^{2} = \frac{s_{11}^{A}(k)s_{22}^{A}(k)-s_{12}^{A}(k)s_{21}^{A}(k)}{|s_{11}^{A}(k)|^{2}} = \frac{s_{33}(k)}{|s_{11}^{A}(k)|^{2}} = \frac{X_{33}(x_{l},k)}{|s_{11}^{A}(k)|^{2}}.
\end{align*}
We will prove that $X_{33}(x,k)>0$ for all $x \in \mathbb{R}$ and for all $k \leq 0$. Define $P = P(k)$ by
\begin{align*}
P = \begin{pmatrix}
\omega & \omega^{2} & 1 \\
\omega^{2}k & \omega k & k \\
k^{2} & k^{2} & k^{2}
\end{pmatrix}
\end{align*}
and define $\widetilde{L} = \widetilde{L}(x,k)$ by
\begin{align*}
\widetilde{L} = P(\mathcal{L} + \mathcal{U})P^{-1} = \begin{pmatrix}
(\omega-\omega^{2})(q_{0}-\bar{q}_{0}) & 1 & 0 \\
0 & (\omega^{2}-1)(q_{0}-\omega \bar{q}_{0}) & 1 \\
k^{3} & 0 & (1-\omega)(q_{0}-\omega^{2} \bar{q}_{0})
\end{pmatrix}.
\end{align*}
Since $X$ satisfies \eqref{Xlax}, the matrix $\widetilde{X}(x,k) := P(k)X(x,k) e^{x\mathcal{L}(k)}$ satisfies
\begin{align}\label{x part of Xtilde}
\widetilde{X}_{x} = \widetilde{L} \widetilde{X}.
\end{align}
Note that 
\begin{align}\label{X33 in terms of f1f2f3}
X_{33} = \frac{1}{3}(f_{1}+f_{2}+f_{3}) e^{-x k},
\end{align}
where
\begin{align*}
f_{1} = \widetilde{X}_{13}, \qquad f_{2} =k^{-1} \widetilde{X}_{23}, \qquad f_{3} = k^{-2} \widetilde{X}_{33}.
\end{align*}
Since the functions $f_{1}$, $f_{2}$, and $f_{3}$ can be rewritten as
\begin{align*}
f_{1} = \sum_{j=1}^{3}\omega^{j}X_{j3}e^{xk}, \qquad f_{2} = \sum_{j=1}^{3}\omega^{2j}X_{j3}e^{xk}, \qquad f_{3} = \sum_{j=1}^{3}X_{j3}e^{xk},
\end{align*}
they are well-defined for $k=0$. Clearly, $X_{33}>0$ is equivalent to $f_{1}+f_{2}+f_{3}>0$. We will show that in fact 
\begin{align}\label{inequality for fj to be proved}
f_{j}(x,k)>0 \qquad \mbox{ for all } j=1,2,3, \; x \in \mathbb{R}, \; k\leq 0.
\end{align}
From \eqref{x part of Xtilde}, we get
\begin{subequations}\label{f1f2f3}
\begin{align}
& k f_{2} = \alpha_{1} f_{1} + f_{1}', & & \alpha_{1}(x) := (\omega^{2}-\omega)q_{0}(x) + (\omega-\omega^{2})\overline{q_{0}(x)} , \label{f2} \\
& k f_{3} = \alpha_{2} f_{2} + f_{2}', & & \alpha_{2}(x) := (1-\omega^{2})q_{0}(x) + (1-\omega)\overline{q_{0}(x)}, \label{f3} \\
& k f_{1} = \alpha_{3}  f_{3} + f_{3}', & & \alpha_{3}(x) := (\omega-1)q_{0}(x) + (\omega^{2}-1)\overline{q_{0}(x)}, \label{f1}
\end{align}
\end{subequations}
together with the normalization conditions
\begin{align}\label{initial conditions}
\lim_{x \to + \infty}f_{j}(x,k)e^{-xk} = 1, \qquad \lim_{x \to + \infty}f_{j}'(x,k)e^{-xk} = k, \qquad \lim_{x \to + \infty}f_{j}''(x,k)e^{-xk} = k^{2},
\end{align}
where prime denotes differentiation with respect to $x$.
If $k=0$, this system decouples and $f_{1},f_{2},f_{3}$ can be solved for explicitly:
\begin{align*}
& f_{1}(x,0) = e^{-\int_{+\infty}^{x}\alpha_{1}}, & &  f_{2}(x,0) = e^{-\int_{+\infty}^{x}\alpha_{2}}, & &  f_{3}(x,0) = e^{-\int_{+\infty}^{x}\alpha_{3}}.
\end{align*}
Since the functions $\alpha_{1}$, $\alpha_{2}$, $\alpha_{3}$ are real-valued, \eqref{inequality for fj to be proved} is proved for $k=0$. Using \eqref{f1f2f3} and the fact that
\begin{align*}
\alpha_{1}+\alpha_{2}+\alpha_{3}=0, \qquad \alpha_{1}\alpha_{2}+\alpha_{1}\alpha_{3}+\alpha_{2}\alpha_{3} = -9|q_{0}|^{2},
\end{align*}
we infer that the functions $f_1, f_2, f_3$ satisfy the following third-order ordinary differential equations
\begin{align}\nonumber
& f_{1}''' + (-9|q_{0}|^{2}+2\alpha_{1}'+\alpha_{2}')f_{1}' + (-k^{3}+\alpha_{1}\alpha_{2}\alpha_{3} + \alpha_{2}\alpha_{1}'+\alpha_{3}\alpha_{1}'+\alpha_{1}\alpha_{2}' + \alpha_{1}'')f_{1} = 0, 
	\\\nonumber
& f_{2}''' + (-9|q_{0}|^{2}+2\alpha_{2}'+\alpha_{3}')f_{2}' + (-k^{3}+\alpha_{1}\alpha_{2}\alpha_{3} + \alpha_{3}\alpha_{2}'+\alpha_{1}\alpha_{2}'+\alpha_{2}\alpha_{3}' + \alpha_{2}'')f_{2} = 0, 
	\\ \label{f3eq}
& f_{3}''' + (-9|q_{0}|^{2}+2\alpha_{3}'+\alpha_{1}')f_{3}' + (-k^{3}+\alpha_{1}\alpha_{2}\alpha_{3} + \alpha_{1}\alpha_{3}'+\alpha_{2}\alpha_{3}'+\alpha_{3}\alpha_{1}' + \alpha_{3}'')f_{3} = 0.
\end{align}

If $\tilde{q}_{0} = \omega q_{0}$, then $\overline{\tilde{q}_{0}} = \omega^2 \bar{q}_{0}$ and hence
\begin{align*}
& \alpha_{1}(x; \omega q_0) = (1-\omega^2)q_{0} + (1-\omega)\bar{q}_{0} = \alpha_{2}(x;q_0),
	\\
& \alpha_{2}(x; \omega q_0) = (\omega-1)q_{0} + (\omega^{2}-1)\bar{q}_{0} = \alpha_{3}(x;q_0), 
	\\
& \alpha_{3} (x; \omega q_0) =  (\omega^{2}-\omega)q_{0} + (\omega-\omega^{2})\bar{q}_{0}  = \alpha_{1}(x;q_0).
\end{align*}
It follows that the equations for $f_2$ and $f_3$ are obtained by replacing $q_{0}$ by $\omega q_{0}$ and $\omega^2 q_{0}$, respectively, in the equation for $f_1$. Moreover, the normalization conditions \eqref{initial conditions} are invariant under these replacements. Thus,
\begin{align}\label{relation between f1f2f3}
f_1(x,k; \omega q_{0}) = f_2(x,k; q_{0}), \qquad f_1(x,k; \omega^2 q_{0}) = f_3(x,k; q_{0}).
\end{align}
We conclude that it is enough to consider one of the three equations. Indeed, if the inequality (\ref{inequality for fj to be proved}) holds for either $j = 1$, $j = 2$, or $j = 3$ for any choice of the complex-valued compactly supported initial data, then it automatically holds for $j = 1$, $j = 2$, and $j = 3$.

Let us consider the equation (\ref{f3eq}) for $f_3$. Fix $k < 0$ and let 
$$g(x) := \frac{f_3(x,k)}{f_3(x,0)}.$$
We know that $f_3(x,0) = e^{-\int_{+\infty}^x \alpha_3(x') dx'}$, and hence it follows that $g$ satisfies the equation
\begin{align}\label{geq2}
g''' + 3p_1g'' + 3p_2g' = k^3 g,
\end{align}
where
\begin{align}
&  p_1(x) := -\alpha_{3}(x) = (1-\omega)q_{0} + (1-\omega^2)\bar{q}_{0},
  	\\
& p_2(x) := 3|q_{0}|^2 - 3\omega q_{0}^2 - 3\omega^2 \bar{q}_{0}^2 - \omega q_{0}' - \omega^2 \bar{q}_{0}'.	
\end{align}
We next solve this equation using the method of variation of constants. 
Let $x_0 \in \R$ be a point to the right of the support of $q_0$. 
Let $y_1(x) \equiv 1$, $y_2(x)$, $y_3(x)$, denote the three linearly independent solutions of the  equation
\begin{align}\label{yeq}
y''' + 3p_1y'' + 3p_2y' = 0
\end{align}
satisfying the initial conditions
\begin{subequations}
\begin{align}
& y_1(x_0) = 1, \quad y_1'(x_0) = 0, \quad y_1''(x_0) = 0, \label{y1 initial}
	\\
& y_2(x_0) = 0, \quad y_2'(x_0) = 1, \quad y_2''(x_0) = 0, \label{y2 initial} 
	\\
& y_3(x_0) = 0, \quad y_3'(x_0) = 0, \quad y_3''(x_0) = 1.  \label{y3 initial}
\end{align}
\end{subequations}
Note that $g(x) = f_3(x,k) = e^{xk} > 0$ for all $x \geq x_{0}$. 
We seek $c_j(x)$, $j = 1,2,3$, such that
\begin{align}\label{gsumcy}
g(x) = \sum_{j=1}^3 c_j(x) y_j(x).
\end{align}
Assume that
\begin{align}\label{cjsystem}
\sum_{j=1}^3 c_j'(x) y_j(x) = 0, \qquad \sum_{j=1}^3 c_j'(x) y_j'(x) = 0, \qquad
\sum_{j=1}^3 c_j'(x) y_j''(x) = k^3 g(x).
\end{align}
Repeated differentiation of (\ref{gsumcy}) gives
$$g'(x) = \sum_{j=1}^3 c_j(x) y_j'(x), \qquad g''(x) = \sum_{j=1}^3 c_j(x) y_j''(x),$$
and
\begin{align*}
g''' & = \sum_{j=1}^3 c_j' y_j'' + \sum_{j=1}^3 c_j y_j'''
= k^3 g - \sum_{j=1}^3 c_j (3p_1y_j'' + 3p_2y_j')
 = k^3 g - (3p_1g'' + 3p_2g'),
\end{align*}
showing that $g$ indeed satisfies (\ref{geq2}).
The equations (\ref{cjsystem}) can be written as 
$$\begin{pmatrix} y_1 & y_2 & y_3 \\
y_1' & y_2' & y_3' \\
y_1'' & y_2'' & y_3''
\end{pmatrix} \begin{pmatrix} c_1' \\ c_2' \\ c_3' \end{pmatrix} = \begin{pmatrix} 0 \\ 0 \\ k^3 g \end{pmatrix}$$
and by Cramer's rule the solution of this linear system is
$$c_1' = \frac{\det \begin{pmatrix} 0 & y_2 & y_3 \\
0 & y_2' & y_3' \\
k^3g & y_2'' & y_3''
\end{pmatrix}}{W(x)}, \quad
c_2' = \frac{\det \begin{pmatrix} y_1 & 0 & y_3 \\
y_1' & 0 & y_3' \\
y_1'' & k^3g & y_3''
\end{pmatrix}}{W(x)}, \quad
c_3' = \frac{\det \begin{pmatrix} y_1 & y_2 & 0 \\
y_1' & y_2' & 0 \\
y_1'' & y_2'' & k^3g
\end{pmatrix}}{W(x)},$$
where
$$W(x) := \det \begin{pmatrix} y_1(x) & y_2(x) & y_3(x) \\
y_1'(x) & y_2'(x) & y_3'(x) \\
y_1''(x) & y_2''(x) & y_3''(x)
\end{pmatrix}.$$
Thus
\begin{align*}
g(x) = &\; \sum_{j=1}^3 c_j(x) y_j(x)
= \sum_{j=1}^3 C_j y_j(x) 
+ \int_{x_0}^x  \frac{k^3g(t)}{W(t)} \det \begin{pmatrix} y_1(x) & y_2(x) & y_3(x) \\
y_1(t) & y_2(t) & y_3(t) \\
y_1'(t) & y_2'(t) & y_3'(t)
\end{pmatrix} dt,
\end{align*}
where the constants $\{C_j\}_1^3$ are determined by the initial conditions
\begin{align}\label{ICscompactcase}
g(x_0) = e^{kx_0} > 0, \quad g'(x_0) = ke^{kx_0} < 0, \quad g''(x_0) = k^2e^{kx_0} > 0.
\end{align}
Using that $y_1 \equiv 1$, we find
$$W(x) = \det \begin{pmatrix} y_2'(x) & y_3'(x) \\ y_2''(x) & y_3''(x)
\end{pmatrix}$$
and
\begin{align}\label{eq for g}
g(x) = &\; C_1 + C_2 y_2(x) + C_3 y_3(x) 
+ \int_{x_0}^x  \frac{k^3g(t)}{W(t)} \det \begin{pmatrix} 1 & y_2(x) & y_3(x) \\ 1 & y_2(t) & y_3(t) \\  0 & y_2'(t) & y_3'(t)
\end{pmatrix} dt.
\end{align}
Since
\begin{align*}
& \frac{d}{dx}\int_{x_0}^x  \frac{k^3g(t)}{W(t)} \det \begin{pmatrix} 1 & y_2(x) & y_3(x) \\ 1 & y_2(t) & y_3(t) \\  0 & y_2'(t) & y_3'(t)
\end{pmatrix} dt	
=  \int_{x_0}^x  \frac{k^3g(t)}{W(t)} \det \begin{pmatrix} 0 & y_2'(x) & y_3'(x) \\ 1 & y_2(t) & y_3(t) \\  0 & y_2'(t) & y_3'(t)
\end{pmatrix} dt	,
	\\
& \frac{d^2}{dx^2}\int_{x_0}^x  \frac{k^3g(t)}{W(t)} \det \begin{pmatrix} 1 & y_2(x) & y_3(x) \\ 1 & y_2(t) & y_3(t) \\  0 & y_2'(t) & y_3'(t)
\end{pmatrix} dt	
=  \int_{x_0}^x  \frac{k^3g(t)}{W(t)} \det \begin{pmatrix} 0  & y_2''(x) & y_3''(x) \\ 1 & y_2(t) & y_3(t) \\  0 & y_2'(t) & y_3'(t)
\end{pmatrix} dt	,
\end{align*}
we find that
\begin{align*}
C_1 = g(x_0) > 0, \qquad C_2 = g'(x_0) < 0, \qquad C_3 = g''(x_0) > 0.
\end{align*}
Abel's identity implies that
\begin{align*}
W(x) = Ce^{-\int_{x_0}^x 3p_1(t) dt}
\end{align*}
for some constant $C$; evaluation at $x = x_0$ gives $C = 1$, and hence
\begin{align*}
W(x) = e^{-\int_{x_0}^x 3p_1(t) dt} > 0.
\end{align*}
We will prove that $g'(x) \leq 0$ for all $x \leq x_{0}$. Using \eqref{eq for g}, we get
\begin{align}
g'(x) = &\; C_2y_2'(x) + C_3y_3'(x) - \int_{x_0}^x  k^3g(t) e^{\int_{x_0}^t 3p_1(t') dt'} \det \begin{pmatrix} y_2'(x) & y_3'(x) \\  y_2'(t) & y_3'(t)
\end{pmatrix} dt	 \label{gprimey}
\end{align}
Recall that $y_{2}'$ and $y_{3}'$ are two solutions of $u'' + 3p_1u' + 3p_2u=0$. We seek to transform this equation to a simpler form by introducing $v$ by 
$$u(x) = e^{F(x)}v(x),$$
where $F$ is some function yet to be determined. 
Letting $f = F'$, we find that $v$ satisfies
\begin{align}\label{v f equation}
v'' + (2f+3p_1)v' + (f^2 + f' + 3fp_1 + 3p_2) v = 0.
\end{align}
More explicitly, the coefficient of $v$ can be written as
$$f^2 + f' + 3fp_1 + 3p_2 = f^2 + f' + 3f((1-\omega)q_{0} + (1-\omega^2)\bar{q}_{0}) + 9|q_{0}|^2 - 9\omega q_{0}^2 - 9\omega^2 \bar{q}_{0}^2 - 3\omega q_{0}' - 3\omega^2 \bar{q}_{0}'.$$
To cancel the terms involving the derivative $q_{0}'$, we choose
$$f = 3\omega q_{0} + 3\omega^2 \bar{q}_{0}.$$
With this choice, we find somewhat remarkably that the coefficient of $v$ in \eqref{v f equation} vanishes identically, which means that $v$ satisfies
$$v'' + (2f+3p_1)v' = 0, \quad \text{i.e.} \quad v'' - 3(\omega^2q_{0} + \omega \bar{q}_{0})v' = 0.$$
It follows that
$$v(x) = C \int_{x_0}^x e^{3\int_{x_0}^{x'} (\omega^2q_{0} + \omega \bar{q}_{0})dt} dx' + D,$$
where $C$ and  $D$ are integration constants, and so
\begin{align*}
y' = e^{F(x)}v = C e^{\int_{x_0}^x 3(\omega q_{0} + \omega^2 \bar{q}_{0})dt}\int_{x_0}^x e^{3\int_{x_0}^{x'} (\omega^2q_{0} + \omega \bar{q}_{0})dt} dx' + De^{\int_{x_0}^x 3(\omega q_{0} + \omega^2 \bar{q}_{0})dt}.
\end{align*}
In particular, since $x_0$ lies to the right of the support of $q_0$, $y'(x_0) = D$ and $y''(x_0) = C$.
Using \eqref{y2 initial} and \eqref{y3 initial}, we obtain the following explicit expressions for $y_2'$ and $y_3'$:
\begin{align*}
& y_2'(x) = e^{\int_{x_0}^x 3(\omega q_{0} + \omega^2 \bar{q}_{0})dt}, 
	\\
& y_3'(x) = e^{\int_{x_0}^x 3(\omega q_{0} + \omega^2 \bar{q}_{0})dt}\int_{x_0}^x e^{3\int_{x_0}^{x'} (\omega^2q_{0} + \omega \bar{q}_{0})dt} dx'.
\end{align*}
Clearly, $y_{2}'(x)>0$ for all $x \in \R$, $y_3'(x) < 0$ for $x < x_0$, and $y_3'(x) > 0$ for $x > x_0$. Furthermore, we obtain
\begin{align*}
\det \begin{pmatrix} y_2'(x) & y_3'(x) \\  y_2'(t) & y_3'(t)
\end{pmatrix} =  e^{\int_{x_0}^x 3(\omega q_{0} + \omega^2 \bar{q}_{0})dt'}
e^{\int_{x_0}^t 3(\omega q_{0} + \omega^2 \bar{q}_{0})dt'}
\int_{x}^t e^{3\int_{x_0}^{x'} (\omega^2q_{0} + \omega \bar{q}_{0})dt'} dx',
\end{align*}
which is clearly $> 0$ for $x < t$ (and $< 0$ for $x > t$). By \eqref{gprimey}, this shows that $g'(x)<0$ whenever $x \leq x_{0}$. Since $g(x) = e^{xk} > 0$ for $x \geq x_{0}$, we conclude that $g(x)> 0$ (and hence also $f_3(x,k) > 0$) for all $x \in \R$. This proves (\ref{inequality for fj to be proved}) for $j = 3$ and completes the proof of (\ref{lol3}), and hence also of (\ref{general upper bounds}).

Let us now consider the general case of $q_0 \in \mathcal{S}(\mathbb{R})$ where $q_0$ is not necessarily compactly supported.
Let $\eta \in C_{c}^{\infty}$ be a cutoff function satisfying $\eta(x)=1$ for $|x|\leq 1$ and $\eta(x)=0$ for $|x|\geq 2$, and for $j \geq 1$, let $\eta_{j}(x) = \eta(x/j)$. Since $q_{0} \in \mathcal{S}(\mathbb{R})$, $(\eta_{j}q_{0})_{j \geq 1}$ is a sequence of smooth functions with compact support converging to $q_{0}$ in $\mathcal{S}(\mathbb{R})$ as $j \to \infty$. As in \cite[Lemma 4.5]{CL}, we have $r_{2}(k; \eta_{j}q_{0}) \to r_{2}(k;q_{0})$ as $j \to \infty$. The inequality in (\ref{lol4}) therefore follows from (\ref{lol3}). This completes the proof of \eqref{general upper bounds unbounded}.

\bigskip
\noindent
{\bf Acknowledgements.} {\it Support is acknowledged from the European Research Council, Grant Agreement No. 682537, the Swedish Research Council, Grant No. 2015-05430, and the Ruth and Nils-Erik Stenb\"ack Foundation.}

\bibliographystyle{plain}
\bibliography{is}

\end{document}